\documentclass[12pt]{article}
\usepackage{amsmath,amssymb,amsfonts,amsthm,url,xspace,graphicx,color}
\newtheorem{theorem}{Theorem}

\numberwithin{theorem}{section}

\newtheorem{lemma}[theorem]{Lemma}
\newtheorem{proposition}[theorem]{Proposition}
\newtheorem{corollary}[theorem]{Corollary}

\theoremstyle{remark}
\newtheorem{rmk}{Remark}
\newtheorem{question}{Question}

\newcommand{\D}{\partial}
\newcommand{\Z}{\mathbb{Z}}
\newcommand{\R}{\mathbb{R}}

\newcommand{\B}{\mathcal{B}}
\newcommand{\E}{\mathbb{E}}
\newcommand{\F}{f\!}

\newcommand{\N}{\mathcal{N}}
\newcommand{\G}{\mathcal{G}}
\newcommand{\Li}{\mathrm{Li}}
\renewcommand{\H}{\mathbb{H}}
\renewcommand{\Re}{\mathrm{Re}}
\renewcommand{\Im}{\mathrm{Im}}
\renewcommand{\arg}{\mathrm{Arg}}
\newcommand{\eps}{\varepsilon}
\newcommand{\be}{\begin{equation}}
\newcommand{\ee}{\end{equation}}
\newcommand{\old}[1]{}

\newcommand{\X}{{\mathcal X}}
\newcommand{\Y}{{\mathcal Y}}

\newcommand{\td}[1]{{\bf\color{blue} #1}}

\begin{document}
\title{The genus-zero five-vertex model}
\author{Richard Kenyon\thanks{Department of Mathematics, Yale University, New Haven CT 06520; richard.kenyon at yale.edu. Research supported by NSF DMS-1854272, DMS-1939926 and the Simons Foundation grant 327929.} \and Istv\'an Prause\thanks{
Department of Physics and Mathematics, University of Eastern Finland, P.O. Box 111, 80101 Joensuu, Finland; istvan.prause at uef.fi.
}}

\date{}
\maketitle
\begin{abstract}
We study the free energy and limit shape problem for the five-vertex model with periodic ``genus zero" weights. We derive the exact phase diagram, free energy and surface tension for this model. We show that its surface tension has trivial potential and use this to give explicit parameterizations of limit shapes.
\end{abstract}

\section{Introduction}
{\let\thefootnote\relax\footnote{{\bf keywords:} square ice model, vertex model, Bethe Ansatz.
\noindent{\bf 2020 MSC:} 82B20}}

A \emph{monotone nonintersecting lattice path configuration} (MNLP configuration) 
on $\Z^2$ is a collection of vertex disjoint,
North- and West-going nearest-neighbor paths in the square grid; see Figure \ref{heights}.
An MNLP configuration has an associated $\Z$-valued \emph{height function} on faces of the graph,
well defined up to an additive constant, as seen in that figure. 
\begin{figure}[tp]
\centerline{\includegraphics[width=1.5in]{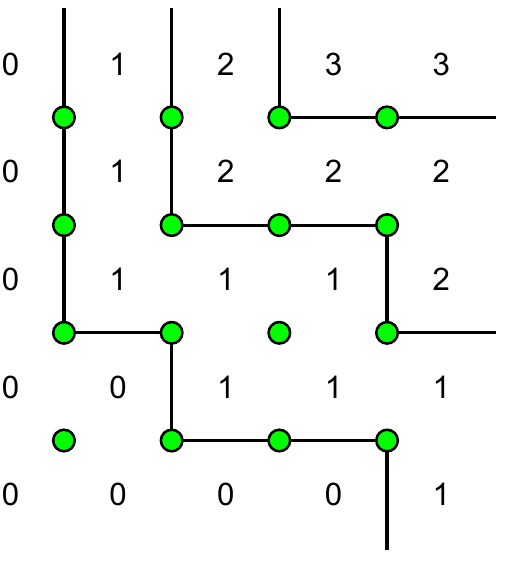}}
\caption{\label{heights}An MNLP configuration and associated height function.
The height function is constant on complementary regions of the paths and increases by $1$ 
when crossing a path upwards or to the right.}
\end{figure}

The \emph{five vertex model} is a probability measure on MNLP configurations (on 
a finite subgraph of $\Z^2$), 
where the probability of a configuration
is proportional to the product of its \emph{vertex weights} as shown in Figure \ref{vtxwts}.
\begin{figure}[htbp]
\centerline{\includegraphics[width=3in]{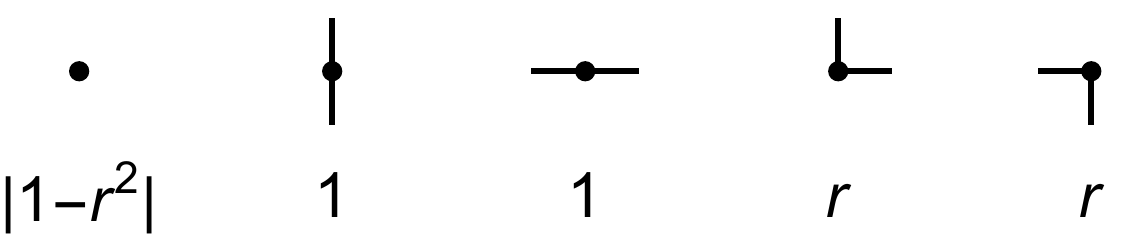}}
\caption{\label{vtxwts} Five-vertex Boltzmann weights. Here $r>0,$ $r \neq 1$ is a parameter of the model.}
\end{figure}

The five vertex model is a special case of the well-known six-vertex model introduced by Pauling \cite{Pauling} and first studied by Lieb \cite{lieb1967exact} in a symmetric case; 
see Nolden \cite{nolden1992asymmetric} for a discussion of the general case.
The five vertex model was studied by Noh and Kim \cite{NohKim}. Recently de Gier, Kenyon and Watson \cite{dGKW} gave a rather complete
study of the five vertex model, including the explicit phase diagram, free energy, surface tension, and limit shapes. 

In this paper we study a generalization of the five-vertex model, defined as above but where the 
quantity $r$ defining the vertex weights can vary from vertex to vertex. 
More specifically, we study the case of a certain subfamily of 
staggered weights (by which we mean vertex weights periodic under a sublattice 
$m_1\Z\oplus m_2\Z$)
which we call the \emph{genus-zero five vertex model}. This nomenclature refers to the fact
that the underlying ``amoeba" is simply connected, see below.
We extend the results of \cite{dGKW} to this genus-zero case. In particular we find explicit expressions
for the phase diagram, free energy, surface tension and limit shapes for Dirichlet boundary conditions.
The novelty here is that these models are
quite difficult to work with directly via the Bethe Ansatz, but we can use certain features
such as \emph{trivial potential} (see \cite{KP1} and below) to give a fairly complete limit shape theory. 
We moreover get a richer family of limit shapes than those for the simply-periodic model.

Our weights are not the most general periodic weights one can impose on the five-vertex model.
Indeed, the fully asymmetric six-vertex model of \cite{nolden1992asymmetric}
can be modeled as a staggered-weight five-vertex model with weights
periodic under $2\Z\times 2\Z$, and we can't handle this case at present. 
It remains to be seen what is the most general subvariety of periodic
weights for the five-vertex model to which our methods apply, see Section \ref{openq}.

\paragraph{The genus-zero five-vertex model.}
Let $\{\alpha_i\}_{i\in\Z}$ and $\{\beta_j\}_{j\in\Z}$ be sequences of positive reals, 
periodic with periods $m_1$ and $m_2$, respectively,
that is $\alpha_{i+m_1}=\alpha_i$ and $\beta_{j+m_2}=\beta_j$ for all $i,j$. We also assume either 
$\alpha_i\beta_j<1$
for all $i,j$ \emph{or} $\alpha_i\beta_j>1$ for all $i,j$. Since the model only depends on the products $\alpha_i\beta_j$,
without loss of generality we can scale all $\alpha$s by a constant and $\beta$s by the inverse constant so that either all $\alpha_i,\beta_j<1$ or all  
$\alpha_i,\beta_j>1$.
We consider the five-vertex model in which at vertex $v=(x,y)$ we have $r_v = \alpha_x\beta_y$, where $r_v$ is the parameter $r$ of Figure \ref{vtxwts}. 
This is a generalization of the simply periodic 
$m_1=m_2=1$ case considered in \cite{dGKW}.  We call this family of weights \emph{genus zero} weights,
and the corresponding measure the \emph{genus-zero five-vertex model}. We refer to the $r_v<1$ case
as the \emph{small $r$ case} and the $r_v>1$ case as the \emph{large $r$ case}.

For a probability measure $\mu$ on MNLP configurations of $\Z^2$, invariant under translations in 
$m_1\Z\oplus m_2\Z$,
let $s = \frac1{m_1}\E[h(m_1,0)-h(0,0)]$ be the expected horizontal height change for one lattice step, and $t=\frac1{m_2}\E[h(0,m_2)-h(0,0)]$ be the expected vertical height change. 
Then $(s,t)$ is said to be the \emph{slope} of $\mu$. The constraints of monotonicity
and disjointness on the lattice paths imply that the
parameters $(s,t)$ range over the triangle $\N\subset\R^2$ with vertices
$\{(0,0),(1,0),(0,1)\}$, which we refer to as the \emph{slope polygon}.

We show that the weights $\{\alpha_i\},\{\beta_j\}$ define for each $(s,t)\in\N$ a
possibly non-unique translation-invariant Gibbs measure
$\mu=\mu_{s,t}$ on MNLP configurations on $\Z^2$.
We compute the associated free energy $F(X,Y)$ and surface tension $\sigma(s,t)$ of this family of measures (see definitions below).
For the ``pure'' phases---the domain $\N'\subset \N$ where $\sigma$ is strictly convex (which is all of $\N$ in the large-$r$ case 
and a strict subset of $\N$ in the small-$r$ case)
we show that $\sigma$ has \emph{trivial potential}, as defined in \cite{KP1}.
As a consequence we can give explicit parameterizations
of \emph{limit shapes} (see Section \ref{limitshapesection}) in terms of analytic functions. In other words
the model is \emph{Darboux integrable} in the sense of \cite{BGH}.

There are two key novelties in the paper: the first is allowing for higher periods. Higher periods in the model give rise to richer phase diagrams and more intricate arctic boundaries. As far as we know, this is the first paper where the exact surface tension and explicit parameterisations of limit shapes are derived for a non-determinantal model allowing for arbitrarily high periods. To accomplish these tasks the essential idea is to work with the correct conformal coordinates (denoted by $z,w$ and $u$ below). The second novelty is the identification of the trivial potential property for five-vertex type interactions in terms of these coordinates. This property greatly simplifies the analysis of limit shapes even in the simply-periodic case, and we expect it to arise in further models as well.
\bigskip

\noindent{\bf Acknowledgements.} We thank Amol Aggarwal, 
Jan de Gier, Vadim Gorin, Andrei Okounkov, and Nicolai Reshetikhin for conversations related to this project. We also 
thank the referees
for comments and corrections.

\section{Background}

\subsection{Measures} 

Let $N$ be a multiple of both $m_1$ and $m_2$ and 
let $\G_N = \Z^2/N\Z^2$ be the $N\times N$ grid on a torus.
Let $\Omega_N$ be the set of MNLP configurations on $\G_N$.

For a configuration $m\in\Omega_N$ let $H_x$ and $H_y$ be the total height change going horizontally
(resp. vertically) around $\G_N$. Let $\Omega_N(H_x,H_y)$ be the subset of $\Omega$ consisting of configurations with these
height changes. Vertex weights of Figure \ref{vtxwts} (with $r=r_v$) define a probability measure $\mu_N(H_x,H_y)$ on $\Omega_N(H_x,H_y)$, giving a configuration a probability proportional to the product of its vertex weights.
For $(s,t)\in \N$ let $\mu_{s,t}$ be any subsequential weak limit measure
\be\label{must}\mu_{s,t}= \lim_{N\to\infty}\mu_N([Ns],[Nt]).\ee
Existence of a unique limit is still open in general, but not essential for the calculations in this paper.  Aggarwal \cite{Aggarwal} showed (in the context of the six-vertex model)
that in certain cases, which we call ``coexistence" cases below, there is no ergodic limit.
He also showed, however, that on the coexistence phase boundary, that is, 
when $(s,t)$ is on the upper right boundary of $\N'$ (see Figure \ref{neutral2X2}),
there is a unique limit measure $\mu_{s,t}$ (by showing that there is a unique ergodic Gibbs
measure of slope $(s,t)$). Conjecturally when $(s,t)\in\N'$ the above limit exists and gives a unique ergodic Gibbs measure. So by a slightly premature use of terminology we call $\overline{\N'}$ the \emph{pure} phases.

\subsection{Free energy}

For further details about the material in this section see \cite{KOS}.

In order to compute properties of the measures $\mu_N(H_x,H_y)$ we impose extra weights on the vertices
as shown in Figure \ref{vtxwtse}, defined using two extra \emph{global}
parameters $X,Y$. 
\begin{figure}[htbp]
\centerline{\includegraphics[width=3in]{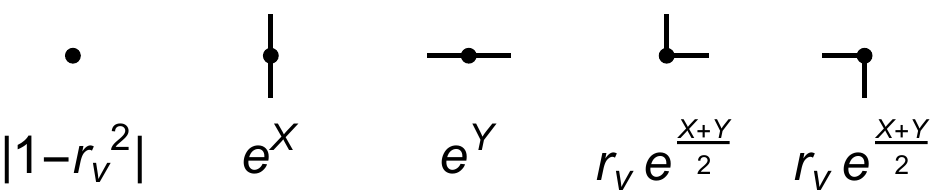}}
\caption{\label{vtxwtse}Vertex weights with fields $X,Y$.}
\end{figure}
Here $(X,Y)\in\R^2$ is referred to as the \emph{magnetic field}. The effect of the magnetic field is to 
weight every vertical edge by $e^X$ and every horizontal edge by $e^Y$.

We let $Z_N=Z_N(\vec\alpha,\vec\beta,X,Y)$ be the partition function, that is, the sum of weights of all MNLP configurations on $\G_N$ (with the vertex weights of Figure \ref{vtxwtse}). Let $\nu_N=\nu_N(\vec\alpha,\vec\beta,X,Y)$ be the natural 
probability measure on $\Omega_N$ assigning a configuration $m$ a probability
proportional to its weight: 
$$\nu_N(m) = \frac{e^{Xh_1+Yh_2}}{Z_N}\prod_{\text{$v$ corners}} r_v\prod_{\text{$v$ empty}}|1-r_v^2|$$
where $h_1=h_1(m),h_2=h_2(m)$ are the number vertical edges and horizontal edges of configuration $m$, the first product is over corners of $m$
and the second is over empty vertices of $m$.
Then $\mu_N(H_x,H_y)$ is the measure $\nu_N$ conditioned to have height change $(H_x,H_y)$, that is, to have $(h_1,h_2)=(NH_x,NH_y)$.

 Since the $X$- and $Y$-dependences of the weight of a configuration
only depend on $H_x,H_y$, $\nu_N(X,Y)$ is a mixture of the measures $\mu_N(H_x,H_y)$:
\be\label{mixture}\nu_N  = \frac1{Z'}\sum_{H_x,H_y} e^{N XH_x+N YH_y}\mu_N(H_x,H_y)\ee
for a constant $Z'=Z'(X,Y)$.

The \emph{free energy} $F(X,Y)=F_{\vec\alpha,\vec\beta}(X,Y)$
is defined to be
$$F(X,Y)= \lim_{N\to\infty}\frac1{N^2}\log Z_N.$$ 
The existence of this limit is a standard subadditivity argument.
The \emph{surface tension}
$\sigma(s,t)=\sigma_{\vec\alpha,\vec\beta}(s,t)$ is the Legendre dual of the free energy $F(X,Y)$,
$$ \sigma(s,t) = \min_{X,Y\in\R}\left(-F(X,Y) + sX+tY\right).$$

If $F$ is smooth near $(X,Y)$ we can define $s=F_X$ and $t=F_Y$
and $\sigma(s,t) = -F(X,Y)+sX+tY$. In this case in the limit of the expression (\ref{mixture})
the RHS concentrates on a single value $(s,t)=\lim_{N\to\infty}\frac1N(H_x,H_y)$. The 
measures $\nu_N(X,Y)$ have subsequential limits which are Gibbs measures of slope $(s,t)$.
Conjecturally (see \cite{Aggarwal}) in this case there is a unique limit 
which is equal to the $\mu_{s,t}$ defined above in (\ref{must}), and is the unique ergodic Gibbs measure of this slope.

If $F$ is not smooth at $(X,Y)$, the graph of $F$ has more than one supporting plane at $(X,Y)$.
Then $(s,t)$, the slope of the support plane, is not uniquely defined. \emph{A priori} 
the measure $\nu_N(X,Y)$ will be some convex combination of ergodic Gibbs
measures of slope $(s,t)$, for $(s,t)$ in the space of
slopes of supporting planes. In this situation we say that this set of slopes is in a \emph{coexistence phase}. A more detailed analysis, which we don't know how to do at present, 
is required to determine which convex 
combination occurs.

\subsection{The simply periodic five-vertex model}

Here we recall facts about the simply periodic five-vertex model, discussed in \cite{dGKW}. 
The simply periodic five-vertex model is the genus-zero five-vertex model in the special case $m_1=m_2=1$.
Let $r=\alpha\beta$. 
Here also there are two distinct cases $r<1$ and $r>1$, which differ in small but important details.

Our definition of vertex weights (Figure \ref{vtxwtse}) differs from those in \cite{dGKW} by the extra factor of $|1-r^2|$
for the ``empty vertex''. This difference can be compensated for by subtracting $\log|1-r^2|$ from $X,Y$; indeed,
dividing all weights of Figure \ref{vtxwtse} by $|1-r^2|$ gives the weights of \cite{dGKW}.
Thus our fields $X,Y$ are related to 
those $\X,\Y$ in \cite{dGKW} by: 
\be\label{newXY}Y = \Y+\log|1-r^2|,~~~~X = \X+\log|1-r^2|.\ee
This causes our formulas for the surface tension and free energy to be slightly different from those in \cite{dGKW}, see below. This is only a notational difference in the simply periodic case; we note that a shift in magnetic fields as in \eqref{newXY} results in a global affine additive term in the surface tension which does not effect limit shapes. The specific choice of weight $|1-r_v|^2$ for the empty vertex will be important, however, when considering higher periods. 

From the definition of the surface tension, at smooth points we have $\nabla \sigma (s,t)=(X,Y)$. The correspondence between the slope $(s,t)$ and the magnetic field $(X,Y)$ is most conveniently described in \emph{microcanonical} (or partial Legendre transform) variables, that is, by the relation between the mixed pairs $(t,X)$ and $(s,Y)$. These pairs are, in turn, encoded in two natural  \emph{conformal coordinates}, $z$ and $w$, which for the simply periodic five-vertex model, satisfy the algebraic relation 
\be\label{P}P(z,w) = 1-z-w+(1-r^2)zw=0.\ee 
This relation defines the ``spectral curve", and holds for both $r<1$ and $r>1$. Here $w$ varies over the upper half plane $\H$ and parameterizes $(s,Y)=(s(w),Y(w))$. At the same time $\bar z \in \H$ parametrizes $(t,X)=(t(\bar z),X(\bar z))$. 
In \cite{dGKW} the simply-periodic $5$-vertex model is solved by the Bethe Ansatz technique, or explicit diagonalization of the (vertical) transfer matrix;
the $z$ coordinate corresponds to the endpoint of the curve of Bethe roots, and the $w$ coordinate corresponds to the endpoint of the 
Bethe root curve for the \emph{horizontal} transfer matrix. The relation (\ref{P}) is mysterious from this point of view, but we know of no easier derivation. 
(The relation becomes slightly less mysterious after having verified in Section \ref{se:trivpot5vtx} that both $z$ and $w$ are intrinsic coordinates in the sense of Proposition \ref{prop:isothermal}. This implies that they are necessarily conformally related to each other.)

The functions $X,Y,s,t$ can be explicitly parameterized as follows (their derivation involves quite different methods than those of the current
paper, so we are using these formulas as a ``black box".)

\paragraph{The $r<1$ case.}
The relevant functions are (cf. \cite{dGKW},\S 4.3-4.4)
\be\label{XYstformulas}X=-\B(\bar z),~~~~Y=-\B(w),~~~~s=\frac{\arg w}{\arg \frac{w}{1-w}},~~~~t=\frac{\arg z}{\arg \frac{z}{1-z}},\ee
where $$\B(z) := \frac1{\pi}(\arg(z)\log|1-z| + \Im\,\Li(z)).$$ 
Here and henceforth we use the principal branch for the argument $\arg(\cdot)$, and  $\Li(z)$ is the dilogarithm function 
$$\Li(z) = -\int_0^z\log(1-z)\frac{dz}{z}.$$


As $w$ varies over the upper half-plane $\H$ it parametrizes the set of slopes in the pure phase $(s,t) \in \mathring{\N'}$
as indicated
in Figure \ref{wsmallr}.  A short calculation show that the limits $w\to \frac{1}{1-r^2},\infty$ blow up to correspond to the two boundary segments of $\N'$,
whereas the intervals when $w$ is in $(-\infty,0),(1,\frac{1}{1-r^2})$ and $(\frac{1}{1-r^2},\infty)$ ``blow down" to the three corners of $\N'$. The interval $w \in (0,1)$ parametrizes the co-existence phase boundary where the limiting value is $s=1-w$.
\begin{figure}[htbp]
\centerline{\includegraphics[width=3.5in]{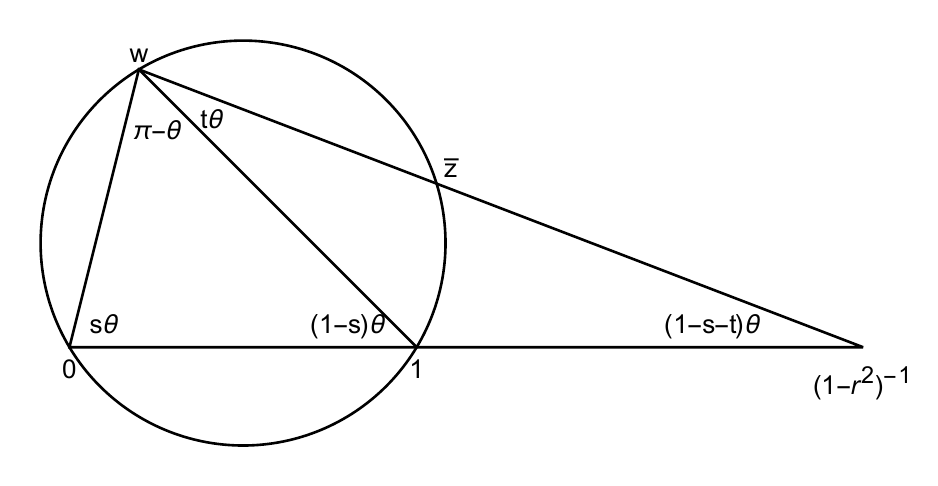}}
\caption{\label{wsmallr}Given $w$ in the upper half plane and $r<1$, define $\theta,s,t,$ and $\bar z$ as indicated, that is, $\theta = \arg\frac{w}{1-w},$ $s=\arg(w)/\theta$, $t=\arg(\bar z)/\theta$ and $z$ satisfies
$P(z,w)=0$.}
\end{figure}

The surface tension is zero in the coexistence phase (bounded by the line $s+t=1$ and
the hyperbola
$1-s-t+(1-r^{-2})st=0$) and in the pure phase it is (cf. \cite{dGKW},Prop. 5.1)
$$\sigma(s,t) = (1-s-t)\B(w) +(1-s)\B(1-(1-r^2)w) + s \B(\frac{(1-r^2)(1-w)}{1-(1-r^2)w}),$$
see Figure \ref{sigsimplesmallr}. 
\begin{figure}
\begin{center}\includegraphics[width=3in]{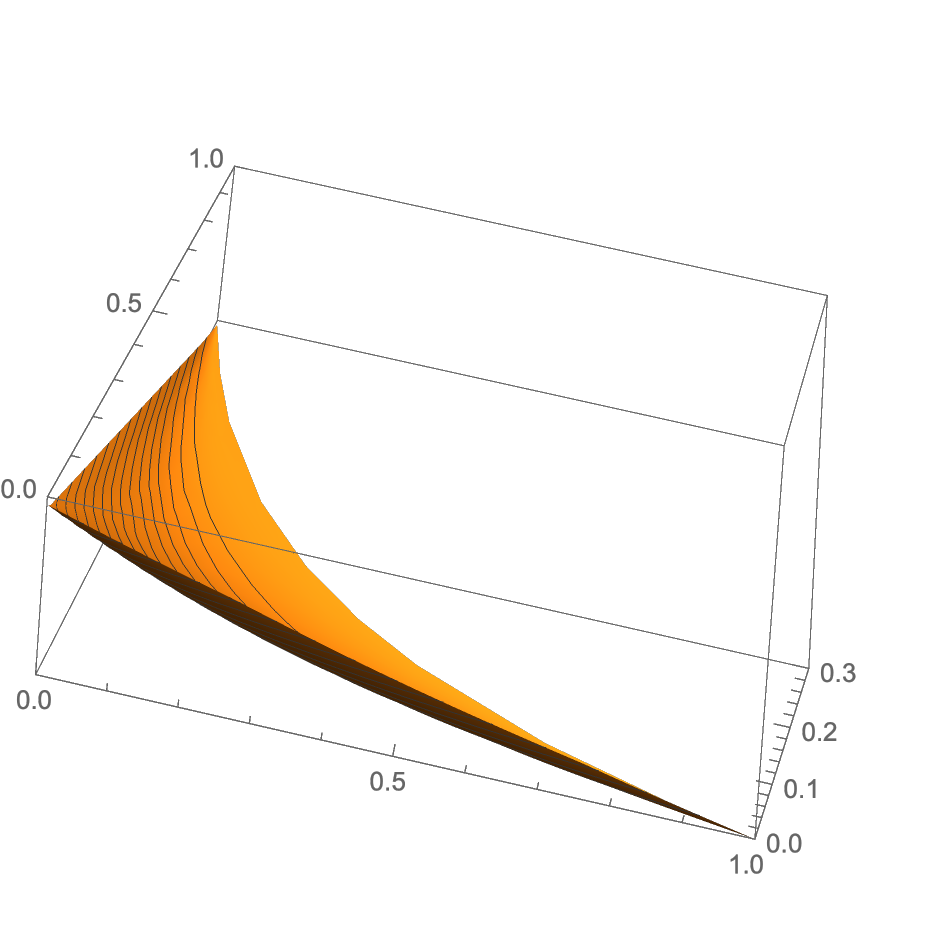}
\hspace{3em}\includegraphics[width=2.5in]{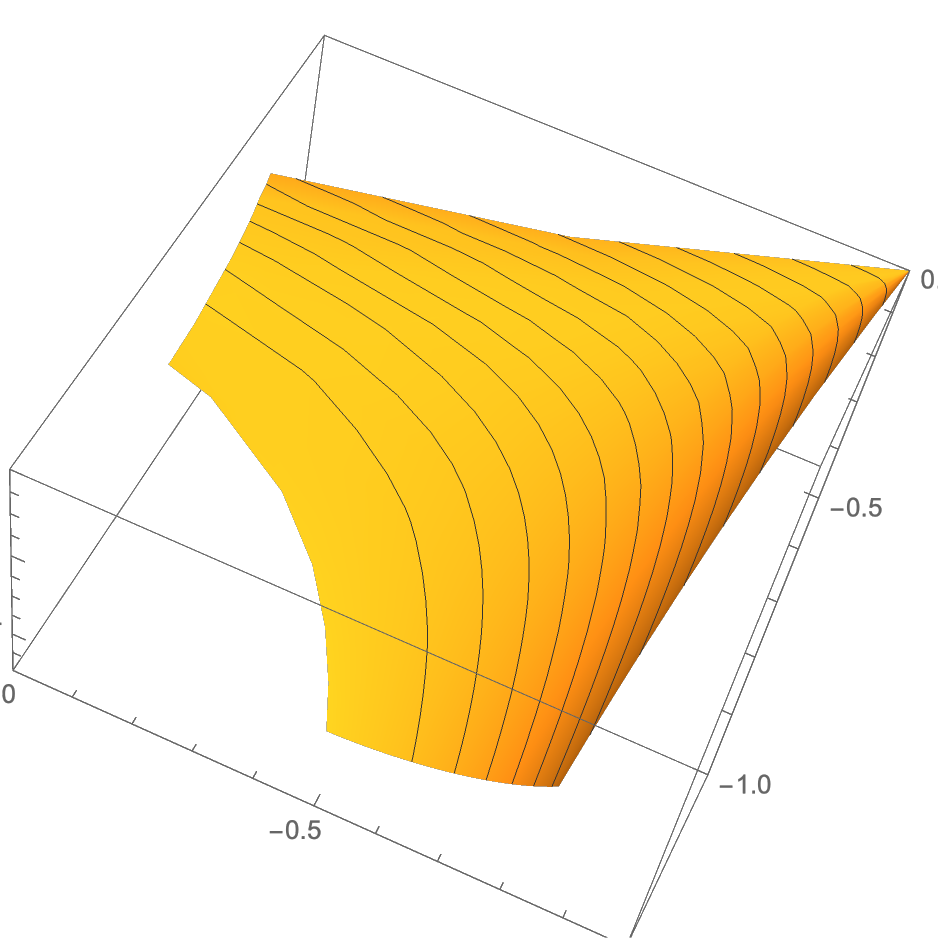}
\end{center}
\caption{\label{sigsimplesmallr} The simply periodic surface tension $\sigma(s,t)$ in the pure phase (left) and free energy $F(X,Y)$ (right) for $r=.5$.}
\end{figure}
Since $w$ and $z$ determine $r$ using (\ref{P}), we can rewrite this using only $z$ and 
$w$:
$$\sigma(s,t) = (1-s-t)\B(w) +(1-s)\B(\frac{1-w}{z}) + s \B(\frac{z+w-1}{w}).$$
This definition of $\sigma$ is symmetric under the simultaneous 
exchange $z\to \bar w, w\to\bar z$ and $s\leftrightarrow t$; this is 
a consequence of the obvious rotational symmetry (using the horizontal transfer matrix rather than the vertical one) or, alternatively,
the 5-term identity for the dilogarithm (see e.g. \cite{zagier2007dilogarithm}) using the fact that $0,1,w,\bar z$ lie on a circle.

 The free energy is 
 \begin{align}F(X,Y) &= -\sigma(s,t) + sX+tY\nonumber\\
 &= -(1-s)\B(w) -(1-s)\B(\frac{1-w}{z}) - s \B(\frac{z+w-1}{w})+s \B(z).\label{freeenergysmallr}\end{align}
 and is plotted in Figure \ref{sigsimplesmallr}.

\paragraph{The $r>1$ case.}
This time the relevant functions are (cf. \cite{dGKW},\S 8.3)
\be\label{XYstformulaslarge}X=\B(\bar z)-\log|z(1-z)|,~~~~Y=\B(w)-\log|w(1-w)|,~~~~
s=\frac{\pi-\arg w}{2\pi-\arg{\frac{w}{1-w}}},~~~~t=\frac{\pi+\arg z}{2\pi+\arg{\frac{z}{1-z}}},\ee
where $\B$ is defined as above.

As $w$ varies over $\H$, it parametrizes the entire set of slopes $(s,t) \in \mathring{\N}$ -- the correspondence is shown in Figure \ref{wbigr}.
\begin{figure}[htbp]
\centerline{\includegraphics[width=5in]{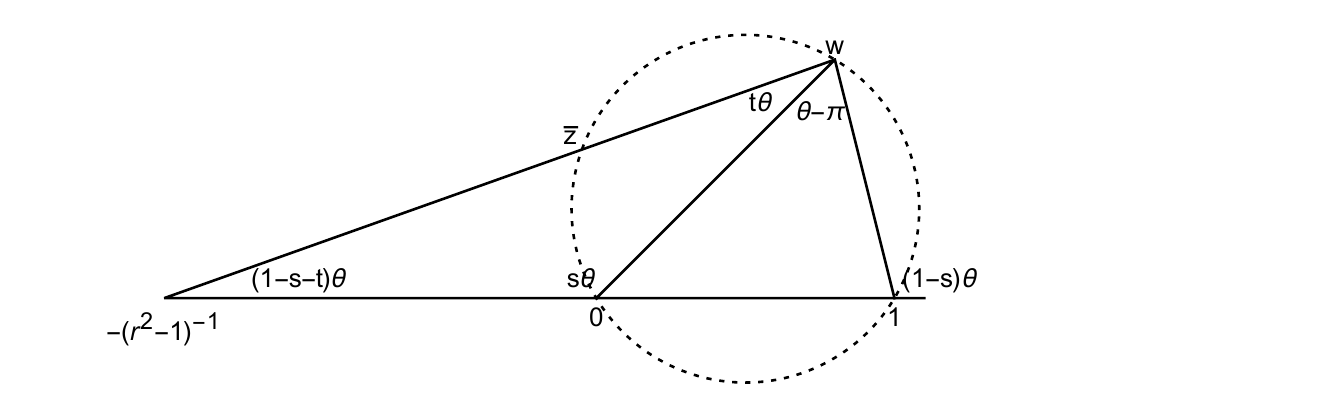}}
\caption{\label{wbigr}Given $w$ and $r>1$, define $\theta,s,t,$ and $\bar z$ as indicated. }
\end{figure}

The surface tension is
$$\sigma(s,t) = (1-s-t)(-\B(w)+\log|w(1-w)|) -(1-s)\B(1+(r^2-1)w) - s \B(\frac{(r^2-1)(w-1)}{1-(1-r^2)w}),$$
see Figure \ref{sigsimplebigr}. 
\begin{figure}
\begin{center}\includegraphics[width=2.8in]{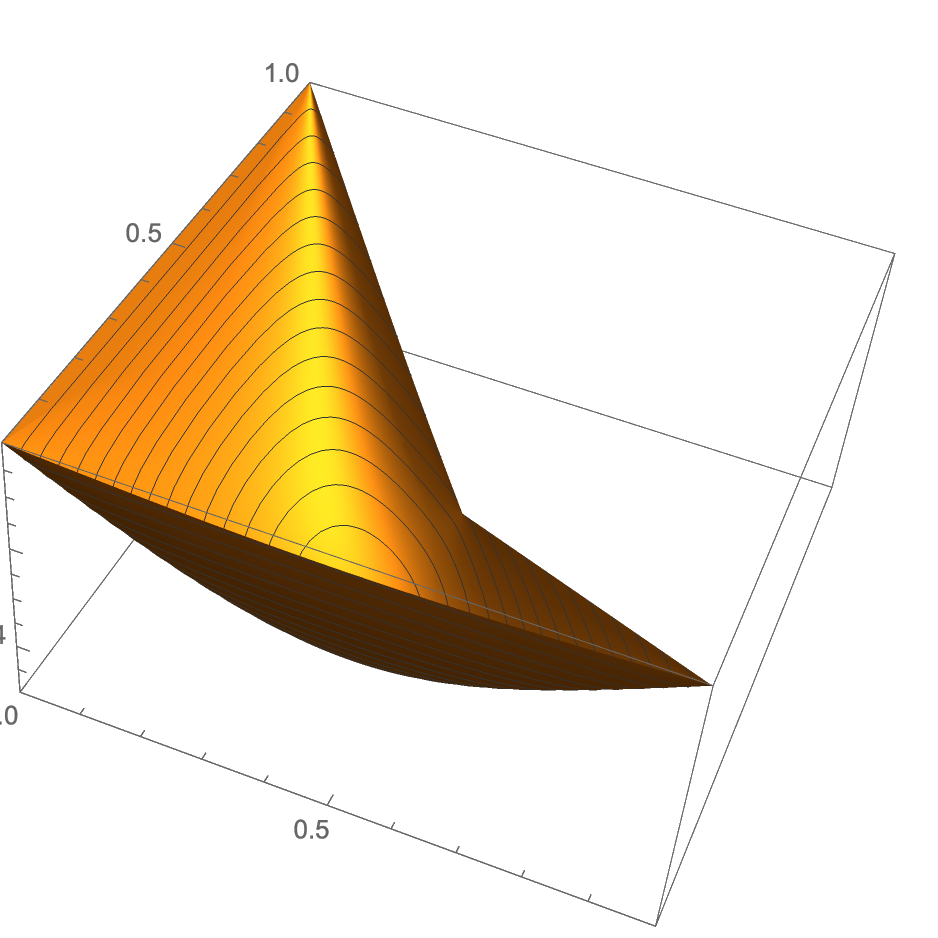}
\hspace{3em}\includegraphics[width=2.8in]{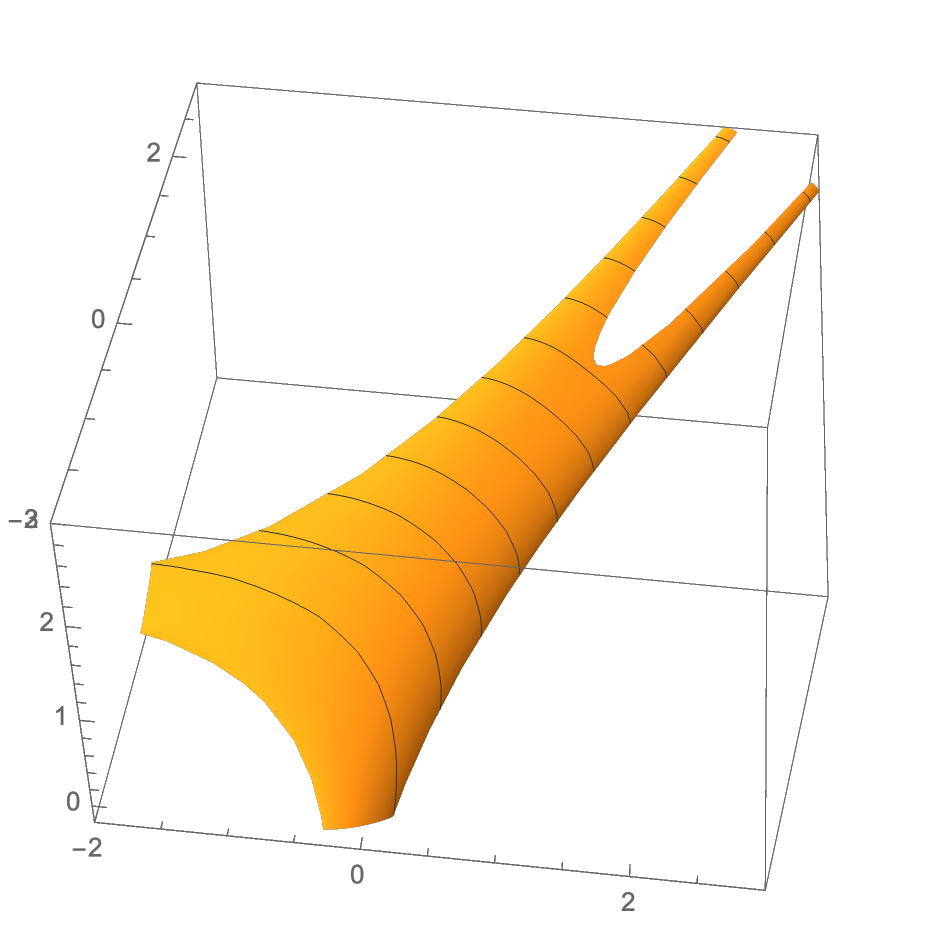}
\end{center}
\caption{\label{sigsimplebigr}The simply periodic surface tension $\sigma(s,t)$ (left) and free energy $F(X,Y)$ (right) for $r=2$.}
\end{figure}
Since $w$ and $z$ determine $r$ using (\ref{P}), we can rewrite this using only $z$ and $w$:
$$\sigma(s,t) = (1-s-t)(-\B(w)+\log|w(1-w)|) -(1-s)\B(\frac{1-w}{z}) - s \B(\frac{z+w-1}{w}).$$

 The free energy is $F(X,Y)=$
 \begin{align}=& -\sigma(s,t) + sX+tY\nonumber\\
 =& (1-s)(\B(w)-\log|w(1-w)|) +(1-s)\B(\frac{1-w}{z}) + s \B(\frac{z+w-1}{w})-s(\B(z)+\log|z(1-z)|).\label{freeenergybigr}\end{align}
 and is shown in Figure \ref{sigsimplebigr}.

\subsection{Trivial potential} 

In \cite{KP1} we introduced a complex variable $\zeta=\zeta(s,t)$ as an \emph{intrinsic coordinate} for a general gradient variational problem and the \emph{trivial potential} property in terms of this variable.
An intrinsic coordinate $\zeta=u+iv$
is a complex coordinate for which the Hessian of the surface tension $\sigma$ in terms of (the real and imaginary parts of) this coordinate is a scalar multiple of the identity: $\sigma_{uu} = \sigma_{vv}, \sigma_{uv}=0$.
Trivial potential is the property that the Hessian determinant of $\sigma$ 
is the fourth power of a harmonic function of the intrinsic variable.
The trivial potential property implies a certain exact parameterization property for solutions, see \cite{KP1}.

As in \cite{KP1} let us consider a smooth convex surface tension $\sigma$ defined on a closed simply-connected $\N' \subset \R^2$ which is strictly convex on $\mathring{\N'}$ with non-zero Hessian determinant. Assume that we are given an orientation reversing diffeomorphism $\zeta \colon \mathring{\N'} \to \H$. With such a change of variable, we can consider both $(s,t)$ and $(X,Y)=(\sigma_s,\sigma_t)$ as functions of the complex variable $\zeta\in\H$. The following proposition from \cite{KP1} then characterizes the intrinsic coordinate. 

\begin{proposition}\cite{KP1}
\label{prop:isothermal}
The following statements are each equivalent to $\zeta$ being an intrinsic coordinate
\begin{enumerate}
\item[(i)]
$ \frac{X_\zeta}{t_\zeta}+\frac{Y_\zeta}{s_\zeta}=0$,
\item[(ii)]
$\frac{s_\zeta}{t_\zeta}=\frac{-\sigma_{st} - i\sqrt{\det H_\sigma}}{\sigma_{ss}}. 
$
\end{enumerate}
and when either of these holds we necessarily have 
$$\frac{X_\zeta}{t_\zeta}=- i \sqrt{\det H_\sigma}=-\frac{Y_\zeta}{s_\zeta}.$$
\end{proposition}

\old{
\begin{proof}
Define $\gamma$ by $\gamma=\frac{s_\zeta}{t_\zeta}$. Since $X_\zeta=(\sigma_s)_\zeta=\sigma_{ss} s_\zeta + \sigma_{st} t_\zeta$ we can write 
$\frac{X_\zeta}{t_\zeta}=\sigma_{ss} \gamma +\sigma_{st}$.
Similarly, 
$\frac{Y_\zeta}{s_\zeta}=\sigma_{tt} \frac{1}{\gamma} +\sigma_{st}$. Now $(i)$ is equivalent to the equation
\[ \sigma_{ss} \gamma^2 +2\sigma_{st} \gamma +\sigma_{tt}=0,
\]
with two solutions
\[ \gamma=\frac{-\sigma_{st} \pm i \sqrt{\det H_\sigma}}{\sigma_{ss}}.
\]
Because $\zeta$ is assumed to be orientation reversing, $\Im \gamma<0$ and thus we have the minus sign above. This leads to $\frac{X_\zeta}{t_\zeta}=-i \sqrt{\det H_\sigma}=-\frac{Y_\zeta}{s_\zeta}$. Thus $(i),(ii)$ and $(iii)$ are all equivalent. On the other hand by definition, \cite{KP1}, the intrinsic complex variable $\zeta$ is an orientation-reversing homeomorphism $\zeta \colon \mathring{\mathcal{N'}} \to \H$ solving the equation
\[
\frac{\zeta_s}{\zeta_t}=-\frac{1}{\bar \gamma}, \quad \text{with} \quad \gamma=\frac{-\sigma_{st}-i\sqrt{\det H_\sigma}}{\sigma_{ss}}.
\]
This is equivalent to $(iii)$ in terms of the inverse mapping.
\end{proof}
}

\subsection{Trivial potential for the $5$-vertex model}
\label{se:trivpot5vtx}
Here we verify that in the case of the simply periodic $5$-vertex surface tension, the conformal coordinates $z,w$ introduced above are intrinsic coordinates,
and show that the surface tension has trivial potential. Later, in Section \ref{se:Hessiandeterminant} we extend this property to the genus-zero model.

For the simply periodic five-vertex model we find from \eqref{XYstformulas} that in the small $r$ case
\begin{align*}
Y_w &= \frac1{\pi}\frac{d}{dw}\left(\frac{(\log w - \log \bar w)}{2i}\frac{(\log(1-w)+\log(1-\bar w))}{2} + \frac{\Li(w)-\Li(\bar w)}{2i}\right)\\
&= \frac1{4\pi i}\left(-\frac{\log(1-w)}{w}+\frac{\log(1-\bar w)}{w}-\frac{\log w}{1-w}+\frac{\log\bar w}{1-w}\right)\\
\end{align*}
and
\begin{align*}
s_w &= \frac{d}{dw}\left(\frac{\log{w}-\log{\bar w}}{\log\frac{w}{1-w}-\log\frac{\bar w}{1-\bar w}}\right)\\
&=-\frac{1}{4\theta^2}\left(\frac1{w}(\log\frac{w}{1-w}-\log\frac{\bar w}{1-\bar w})-(\log{w}-\log{\bar w})(\frac1w+\frac1{1-w})\right)\\
&=-\frac{1}{4\theta^2}\left(-\frac{\log(1-w)}{w}+\frac{\log(1-\bar w)}{w}-\frac{\log w}{1-w}+\frac{\log\bar w}{1-w}\right)
\end{align*}
where $\theta=\arg\frac{w}{1-w}.$
A similar calculation holds for $X_z$ and $t_z$. So we find
\be\label{Yovers}
 \frac{Y_w}{s_w}=\frac{i}{\pi} \left( \arg \frac{w}{1-w}\right)^2=\frac{i}{\pi} \theta^2=\frac{i}{\pi} \left( \arg \frac{z}{1-z}\right)^2=-\frac{X_z}{t_z}=-\frac{X_w}{t_w}.
\ee
In view of Proposition \ref{prop:isothermal}, this means that the $z,w$ coordinates are intrinsic coordinates. Furthermore, by the same proposition
\[ \kappa(w):=\sqrt{\det H_\sigma}=\frac{\theta^2}{\pi}=\frac{1}{\pi} \left( \arg \frac{w}{1-w}\right)^2.
\]
Since, $\kappa^{1/2}(w)$ is a harmonic function of $w$, we conclude that the surface tension $\sigma$ has trivial potential in the pure phase $\N'$. 

Similarly, in the large $r$ case we find from \eqref{XYstformulaslarge}
\[ \frac{Y_w}{s_w}=\frac{i}{\pi} \theta^2=-\frac{X_z}{t_z}=-\frac{X_w}{t_w},
\]
and hence $z$ and $w$ are intrinsic coordinates, and this time
\[ \kappa(w)=\sqrt{\det H_\sigma}=\frac{\theta^2}{\pi}=\frac{1}{\pi} \left(2\pi- \arg \frac{w}{1-w}\right)^2.
\]
Again we find that $\kappa^{1/2}(w)$ is a harmonic function. This means that $\sigma$ has trivial potential in all of $\N$.

\section{Commuting transfer matrices}

We now consider the general genus-zero five-vertex model, with periodic weights $r_v$.
In the small $r$ case, we can remove the absolute value sign from the weight $|1-r_v^2|$, replacing it $1-r_v^2$.
Likewise in the large $r$ case, we replace $|1-r_v^2|$ with $r_v^2-1$. 

\begin{lemma}\label{Tcomm} Let $T_\beta$ be the (vertical) transfer matrix associated to a row of weight $\beta$.
In either the small $r$ case or the large $r$ case,
$T_{\beta}$ and $T_{\beta'}$ commute.
\end{lemma}

There is a corresponding statement, rotating by $90^{\circ}$, with ``horizontal'' transfer matrices replacing ``vertical" transfer matrices.

\begin{rmk}
One can in principle prove this lemma using the Yang-Baxter equation. Indeed, the model is a degeneration of a staggered six-vertex model with $\Delta \to \pm \infty$, and the $6$-vertex model with fixed $\Delta$ has a well known Yang-Baxter equation \cite{Baxter}. The details of the YB equation are somewhat involved, however, and we found that proving the lemma directly was simpler.
\end{rmk}
 
\begin{proof} We assume we are in the small $r$ case; the large $r$ case is identical
replacing each $1-r_v^2$ with $r_v^2-1$ below. 

We compute the matrix element 
$\langle \vec z|T_{\beta'}T_\beta|\vec x\rangle$ and show that it is symmetric in $\beta$ and $\beta'$ for any $\vec x,\vec z$.
The situation is as illustrated in Figure \ref{comm}.
\begin{figure}
\begin{center}\includegraphics[width=5in]{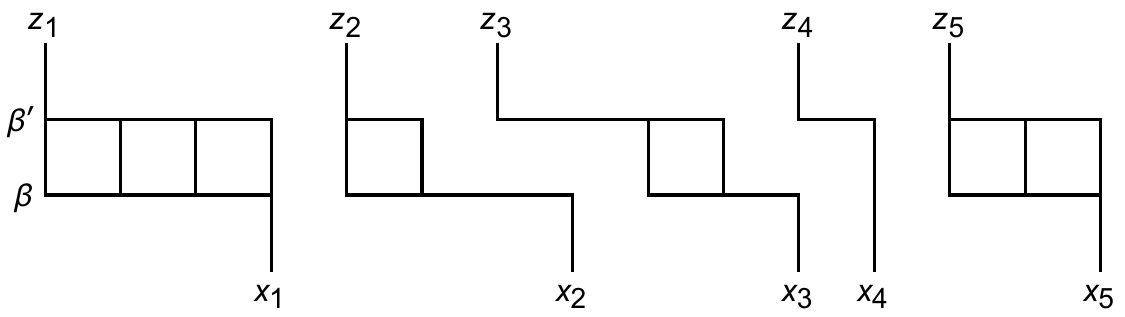}\end{center}
\caption{\label{comm}Calculation for commutation of $T_\beta,T_{\beta'}$.}
\end{figure}
Each NW path from $x_i$ must end at $z_i$ while avoiding adjacent paths.
That is, each path is a NW path in the illustrated component. 
The matrix element $\langle \vec z|T_{\beta'}T_\beta|\vec x\rangle$ is the product over all components, of the contribution  of each component.
We claim that the weight of each component is a symmetric polynomial in $\beta,\beta'$ except for a monomial
which becomes symmetric when paired with the corresponding monomials for the adjacent components.
As an example, consider the weight of the third component (containing $x_3$ and $z_3$). 
Suppose for simplicity that $z_3$ is at $x$-coordinate $0$.
There are two possible paths in this component, and their weights are
$$r_{0\beta'}r_{2\beta'}r_{2\beta}r_{4\beta}(1-r_{3\beta'}^2)$$
and
$$r_{0\beta'}r_{3\beta'}r_{3\beta}r_{4\beta}(1-r_{2\beta}^2).$$
Their sum is
$$r_{0\beta'}r_{4\beta}(\alpha_2^2\beta\beta'(1-\alpha_3^2(\beta')^2)+\alpha_3^2\beta\beta'(1-\alpha_2^2\beta^2))$$
which is symmetric.

To deal with the general situation, there are several cases: the case like that of $x_1$ in the figure, where there are no
``dangling legs", or cases where there is a dangling leg to the left, or to the right (like for $x_2$), or both (like for $x_3$). Note that a component has a dangling leg on 
the left if and only if the component to its left has a right dangling leg.
Since a dangling leg on the left contributes a factor $\beta'$ and one on the right contributes a factor $\beta$,
these can be multiplied in pairs to give a symmetric monomial. 

It therefore suffices to consider a component like that containing $x_3$ (but of general length). 
For a case like $x_3$ where the component has length $n$, the calculation involves showing that the polynomial
\be\label{Apoly}\sum_{k=0}^nA_k\left(\prod_{i=0}^{k-1}(1-A_i x)\right)\left(\prod_{i=k+1}^n(1-A_iy)\right)\ee
is symmetric in $x$ and $y$ (where we took $A_i = \alpha_i^2$, $x=\beta$ and $y=\beta'$).

However one can see that in (\ref{Apoly}) the coefficient of $x^iy^j$ is $(-1)^{i+j}e_{i+j+1}$ where $e_k$ is the elementary symmetric polynomial in the $A_0,\dots,A_n$. The symmetry follows.
\end{proof}

\subsection{Free energy}

The free energy $F_r(X,Y)$ for the simply periodic model (with weight $r$) is given in equations \eqref{freeenergysmallr} and \eqref{freeenergybigr}.
For the periodic model, we can calculate the free energy $F_{\vec\alpha,\vec\beta}(X,Y)$ as follows.
The computation is the same in both the small $r$ and large $r$ cases except for minor modifications.

Recall the \emph{microcanonical free energy} $\F_r(s,Y)$ which is the free energy in the simply periodic case when we restrict
the number $n$ of particles/paths to have horizontal density $s$; 
it is the limit as $n,N\to\infty$ (with $n/N\to s$) of the log of the normalized 
leading eigenvalue of the transfer matrix $T=T_{r,Y}(n,N)$:
$$\F_r(s,Y) = \lim_{N\to\infty}\frac1N\log\Lambda(r,n,N,Y).$$

\paragraph{The $m_1=1$ case.}
Suppose first that $m_1=1$, and set $\alpha=\alpha_1$. For fixed $n,N$, the transfer matrices $T_r,T_{r'}$ commute, by Lemma \ref{Tcomm}. 
We therefore find that applying one vertical period of (vertical) transfer matrices, the eigenvalues multiply and we get
$$\F_{\alpha,\vec\beta}(s,Y) = \frac1{m_2}\sum_{j=1}^{m_2} \F_{\alpha,\beta_j}(s,Y).$$ 

From $\F_{\alpha,\vec\beta}(s,Y)$ we can take the Legendre transform in $Y$ to get 
$\sigma_{\alpha,\vec\beta}(s,t)$:
$$\sigma_{\alpha,\vec\beta}(s,t) = -\F_{\alpha,\vec\beta}(s,Y) +Yt$$
where $t=\frac1{m_2}\sum_{j=1}^{m_2} t_j$ and
$t_j = \frac{\partial}{\partial Y}\F_{\alpha,\beta_j}(s,Y)$ is the vertical slope associated to row $j$.
Thus
\be\label{sigmahalfway}\sigma_{\alpha,\vec\beta}(s,t) = \frac1{m_2}\sum_{j=1}^{m_2} \sigma_{\alpha,\beta_j}(s,t_j).\ee
Despite appearances, the RHS of (\ref{sigmahalfway}) is indeed a function of $t$,  although the dependence on $t$ is not explicit: $t_j$ is a function of $t$ given
by $t_j = \frac{\partial}{\partial Y}\F_{\alpha,\beta_j}(s,Y)$ where $Y$ is determined implicitly by 
$\frac{\partial}{\partial Y} \F_{\alpha,\vec\beta}(s,Y)=t.$
 
We can likewise apply the Legendre transform in $s$ of $\F_{\alpha,\vec\beta}(s,Y)$ to get 
\begin{align*}
F_{\alpha,\vec\beta}(X,Y) &=-\F_{\alpha,\vec\beta}(s,Y)+Xs\\
&= \frac1{m_2}\sum_{j=1}^{m_2} F_{\alpha,\beta_j}(X_j,Y)
\end{align*}
where $X=\frac1{m_2}\sum_{j=1}^{m_2} X_j$ and $X_j = \frac{\partial}{\partial s}\F_{\alpha,\beta_j}(s,Y).$

\paragraph{General $m_1$ case.}

For each column $i$ we have $\sigma_{\alpha_i,\vec\beta}(s,t)$ from (\ref{sigmahalfway}) above.  
Applying the Legendre transform in $s$ we get
$\F_{\alpha_i,\vec\beta}(X,t).$ This quantity is the limit of the logarithm of the leading eigenvalue
of the \emph{horizontal} transfer matrix with horizontal weight $\alpha_i$ and periodic vertical weights 
$\vec\beta$. 
We have 
$$\F_{\alpha_i,\vec\beta}(X,t) = -\sigma_{\alpha_i,\vec\beta}(s,t) + X s,$$
where 
\be\label{Xi}X = \frac{\partial}{\partial s}\sigma_{\alpha_i,\vec\beta}(s,t) = \frac1{m_2}\sum_{j=1}^{m_2}\frac{\partial}{\partial s} \sigma_{\alpha_i,\beta_j}(s,t_j).\ee

Now we can apply the commutation of these horizontal transfer matrices for different $\alpha_i$'s to get
$$\F_{\vec\alpha,\vec\beta}(X,t) = \frac1{m_1}\sum_{i=1}^{m_1}\F_{\alpha_i,\vec\beta}(X,t).$$
One more Legendre transform gives either the free energy
$F_{\vec\alpha,\vec\beta}(X,Y)$ or surface tension $\sigma_{\vec\alpha,\vec\beta}(s,t).$

We can simplify the computation by working with the conformal parameters.
Given $X,Y,\vec\alpha,\vec\beta$, we define $s_i$ and $t_j$ by the equations
 $\frac{\partial}{\partial s}\sigma_{\alpha_i,\vec\beta}(s_i,t)=X$ 
and
$\frac{\partial}{\partial t}\sigma_{\vec\alpha,\beta_j}(s,t_j)=Y$
respectively.  From $s_i,t_j$ and $r_{ij}=\alpha_i\beta_j$ we can find $w_i,z_j$ (from the formulas for
the simply periodic case (\ref{XYstformulas})), satisfying
\be\label{zwij}
1-z_j-w_i+(1-r_{ij}^2)z_jw_i=0.
\ee 
We can rewrite this as
$$\alpha_i^2\beta_j^2 = \frac{(1-z_j)(1-w_i)}{z_jw_i}$$
or
\be\label{zwfromu} \frac{\alpha_i^2w_i}{1-w_i}=u=\frac{1-z_j}{\beta_j^2z_j}.\ee
We see that indeed $z_j$ only depends on $j$ and $w_i$ only depends on $i$
(and the parameter $u \in \H$ is independent of both $i$ and $j$).
Thus we have associated
a particular $w$ value, $w_i$, to each column and a particular $z$ value $z_j$ to each row.
These values in turn define $X_j=X(z_j)$ and $Y_i=Y(w_i)$ via the formulas for the simply periodic case \eqref{XYstformulas}.

Now note that in (\ref{Xi}), with $s=s_i$, the expression $\frac{\partial}{\partial s} \sigma_{\alpha_i,\beta_j}(s_i,t_j)$
is independent of $i$, and equal to $X_j=-\B(\bar z_j)$. Thus $X=\frac1{m_2}\sum_j X_j=-\frac1{m_2}\sum_{j=1}^{m_2} \B(\bar z_j).$
Symmetrically we find $Y=\sum_i Y_i=-\frac1{m_1}\sum_{i=1}^{m_1} \B(w_i)$.
Moreover $t= \frac1{m_2}\sum t_{j}$ and
$s = \frac1{m_1}\sum s_{i}$.

We can finally reconstruct $\sigma_{\vec\alpha,\vec\beta}(s,t)$ as a sum (really the average) of the individual contributions 
$\sigma_{\alpha_i,\beta_j}(s_i,t_j)$,
and likewise $F_{\vec\alpha,\vec\beta}(X,Y)$ is the average of the individual $F_{\alpha_i,\beta_j}(X_i,Y_j)$,
where $s_i, t_j,X_i,Y_j$ are determined from $z_j,w_i$ as in the simply periodic case (\ref{XYstformulas}),(\ref{XYstformulaslarge}).

In the small $r$ case,
$$\sigma(s,t) = \frac1{m_1m_2}\sum_{i=1}^{m_1}\sum_{j=1}^{m_2} (1-s-t)\B(w_i)+(1-s)\B(1-(1-r_{ij}^2)w_i)+s\B(\frac{(1-r_{ij}^2)(1-w_i)}{1-(1-r_{ij}^2)w_i}).$$
Using $1-r_{ij}^2=\frac{z_j+w_i-1}{z_jw_i}$ this is
\be\label{doublesumsig}\sigma(s,t) = \frac1{m_1m_2}\sum_{i=1}^{m_1}\sum_{j=1}^{m_2} (1-s-t)\B(w_i)+(1-s)\B(\frac{1-w_i}{z_j})+s\B(\frac{z_j+w_i-1}{w_i}).\ee

The free energy is 
$$F_{\vec\alpha,\vec\beta}(X,Y) = \frac1{m_1m_2}\sum_{i=1}^{m_1}\sum_{j=1}^{m_2}-(1-s)\B(w_i)-
(1-s)\B(\frac{1-w_i}{z_j})-s \B(\frac{z_j+w_i-1}{w_i}) + s\B(z_j).$$

The corresponding expressions in the large $r$ case are similar and read as follows.

\be\label{sigmabigrgeneral}\sigma(s,t) = \frac1{m_1m_2}\sum_{i=1}^{m_1}\sum_{j=1}^{m_2} (1-s-t)(-\B(w_i)+\log|w_i(1-w_i)|) -(1-s)\B(\frac{1-w_i}{z_j}) - s \B(\frac{z_j+w_i-1}{w_i}) \ee

{\footnotesize$$F_{\vec\alpha,\vec\beta}(X,Y) = \frac1{m_1m_2}\sum_{i=1}^{m_1}\sum_{j=1}^{m_2}
-(1-s)(-\B(w_i)+\log|w_i(1-w_i)|) +(1-s)\B(\frac{1-w_i}{z_j}) + s \B(\frac{z_j+w_i-1}{w_i})+s(\B(z_j)-\log|z_j(1-z_j)|).$$}

\subsection{Hessian determinant.}
\label{se:Hessiandeterminant}

Recall equation \eqref{zwfromu}
\[ \frac{\alpha_i^2w_i}{1-w_i}=u=\frac{1-z_j}{\beta_j^2z_j}\]
for a parameter $u \in \H$ in the upper half plane independent of $i,j$. 
Note that the angle $\theta$ is the same at every vertex in the fundamental domain. Namely, $\theta=\arg u$ in the small $r$ case and $\theta=2\pi-\arg u$ in the large $r$ case. 

For each term in the double sum (\ref{doublesumsig}) as well as in \eqref{sigmabigrgeneral}
for $\sigma$, any one of $z_j$, $w_i$ or $u$ is a conformal coordinate. 
We will show that $u$ is an intrinsic coordinate for the entire sum. We have $$\frac{X_u}{t_u} = \frac{\sum_{j=1}^{m_2}(X_j)_u}{\sum_{j=1}^{m_2}(t_j)_u}.$$
Since $u$ and each $z_j$ are related by a conformal automorphism, using (\ref{Yovers}) for each $j$,
$$\frac{(X_j)_u}{(t_j)_u} = \frac{(X_j)_{z_j}}{(t_j)_{z_j}}  = -\frac{i}{\pi}\theta^2.$$
We conclude that
$\frac{X_u}{t_u} =-\frac{i}{\pi}\theta^2,$
and a similar argument gives  $\frac{Y_u}{s_u} =\frac{i}{\pi}\theta^2.$
Thus by Proposition \ref{prop:isothermal} $u$ is an intrinsic coordinate
(except for the homeomorphism requirement, which we verify below), and
$$\kappa(u) := \sqrt{\det H_\sigma} = \frac{\theta^2}{\pi}.$$
We readily see that the function $\kappa^{1/2}(u)$ is a harmonic function: the model thus has trivial potential.

\old{
The hessian matrix for any of these \td{Explain...}
is $\theta^2/\pi$ times the identity matrix.  Upon averaging, we find that $\sigma(s,t)$ has Hessian 
which is $\theta^2/\pi$ times the identity matrix. Thus the Hessian determinant of $\sigma(s,t)$
is again $\theta^4/\pi^2$: the model
has trivial potential with $u$ being a conformal coordinate.}

The parametrization of the slope variables $(s,t)$ in terms of $u \in \H$ are given by $s=\frac1{m_1} \sum_i s_i$, $t=\frac1{m_2} \sum_j t_j$. Explicitly,
\begin{align}
\label{eq:slopesfromu}
s(u)= \frac{1}{m_1 \theta} \sum_i \arg \frac{u}{u+\alpha_i^2}, \quad t(u)=\frac{1}{m_2 \theta} \sum_j  \arg (1+\beta_j^2 u)  \quad \text{(small $r$ case)},\\
s(u)=\frac{1}{m_1 \theta}\sum_i \left( \pi-\arg \frac{u}{u+\alpha_i^2} \right), \quad t(u)=\frac{1}{m_2 \theta}\sum_j  \left( \pi-\arg (1+\beta_j^2 u) \right) \quad \text{(large $r$ case)}. \nonumber
\end{align}

In both cases the map $u\mapsto(s,t)$ is bijective from $\H$ to $\mathring{\N}'$. To see this,
we note that as functions of $u$, $s_u/t_u=\gamma$ where $\gamma$ has negative imaginary part 
(see Proposition \ref{prop:isothermal}). This implies that the Jacobian determinant of $u\mapsto(s,t)$ is (strictly) negative. 
Thus the map is open except for critical points.
However we claim that the boundary values of $(s,t)$
wind once around $\mathring{\N}'$ as $u$ runs over the boundary of $\H$. To see this, suppose first that we are in the large $r$ case and 
(without loss of generality) $\alpha_i,\beta_j>1$ for all $i,j$, and after reindexing $\alpha_1\le\dots\le\alpha_{m_1}$ and $\beta_1\le\dots\le \beta_{m_2}$. Consider $s,t$ as $u$ runs over the boundary of a large semidisk: the region between the line
$y=\eps$ for small $\eps$ and the disk of radius $R\gg 1$ centered at $0$. 
When $u=u_1+i\eps$, as $u_1$ increases from $-R$, $s(u)$ starts close to $1$, decreases from $1$ to near $0$ on the range $[-\alpha_{m_1}^2,-\alpha_1^2]$, then increases to near $\frac12$ as $u_1$ passes $0$,
and remains close to $\frac12$ for $u_1>0$. Then when $u=Re^{i\theta}$ for large $R$, as $\theta$ runs from $0$ to $\pi$,
$s(u)$ increases from $1/2$ to $1$.
Likewise $t(u)$ is close to $0$ for $u_1<-1/\beta_{1}^2$, increases from $0$ to near $1$ on the range 
$[-1/\beta_{1}^2,-1/\beta_{m_2}^2]$, then decreases to near $\frac12$ as $u_1$ passes $0$,
and remains close to $\frac12$ for $u_1>0$. When $u=Re^{i\theta}$ for large $R$, as $\theta$ runs from $0$ to $\pi$,
$t(u)$ decreases from $1/2$ to $0$. Thus as $\eps\to0$ and $R\to\infty$ the image of $s,t$ is the whole triangle $\N$.
Similarly for $r<1$ we can assume $\alpha_i,\beta_j<1$. Then $s(u_1+\eps i)$ is small at $u_1\ll0$, increases to $1$ on the range
$[-\alpha_{m_1}^2,-\alpha_1^2]$, and then decreases to $0$ on the range $u\in(0,\infty)$.
Also $t(u_1+\eps i)$ is close to $1$ for $u_1\ll0$, decreases to $1$ on the range
$[-1/\beta_{1}^2,-1/\beta_{m_2}^2]$, and then increases to $0$ on the range $u\in(0,\infty)$.
The behavior of $s(u),t(u)$ on the range $(0,\infty)$ is described in the next section.

These calculations prove the claim and imply that there are no critical points, and the map $u\mapsto(s,t)$ is an (orientation-reversing) homeomorphism.

\subsection{Phase diagram}

The phase diagram of the system is displayed in Figure \ref{5vtx2X2amoeba} in terms of the magnetic fields $(X,Y)$ (for $(m_1,m_2)=(2,2)$). There is an amoeba-shaped disordered region where the free energy $F_{\vec\alpha,\vec\beta}(X,Y)$ is strictly convex surrounded by frozen phases in each of which the free energy is affine. We use the terminology ``amoeba'' for the disordered phase in a loose sense; it should be noted that these are \emph{not} amoebae in the algebraic sense.

\begin{figure}[htbp]
\begin{center}
\includegraphics[width=2.7in]{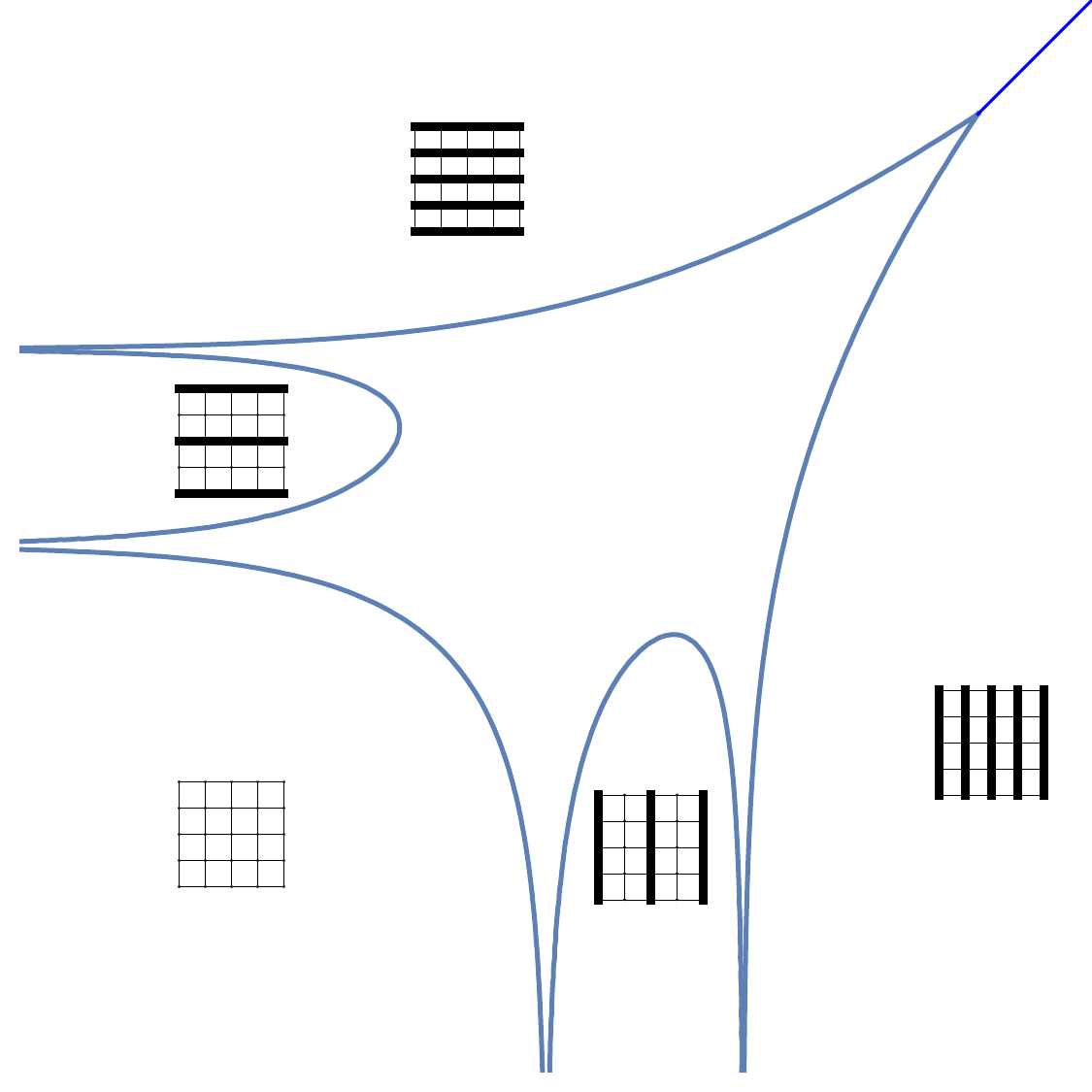} \hspace{2em}
\includegraphics[width=3.in]{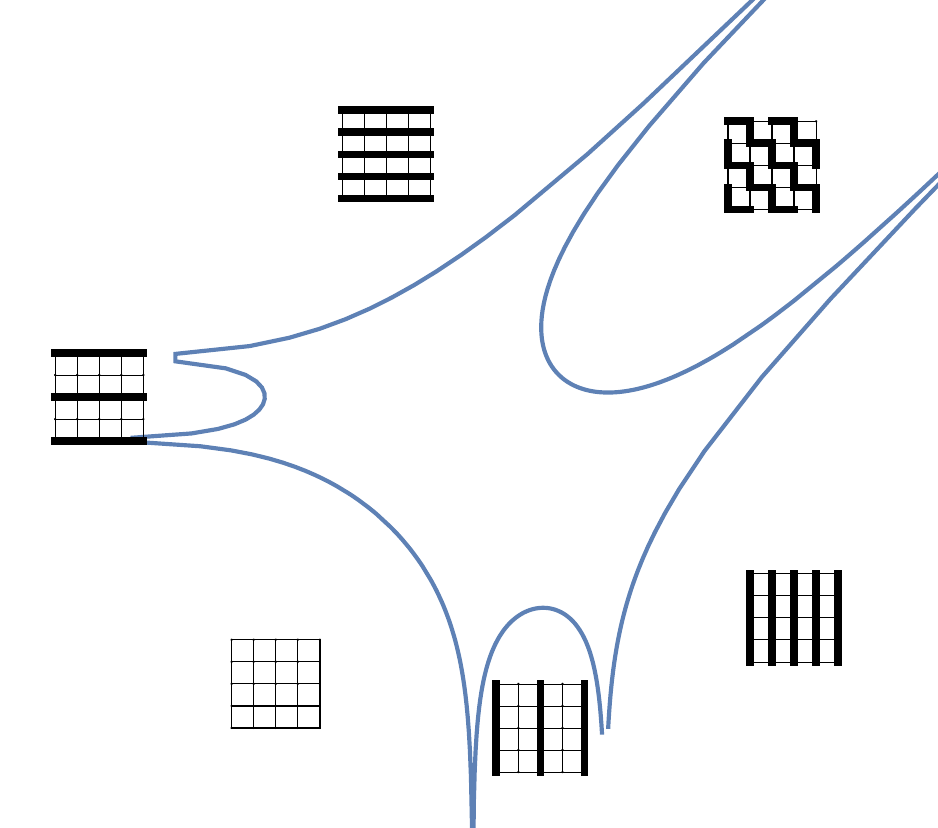}
\end{center}
\caption{\label{5vtx2X2amoeba}Amoeba and external phases for $2\times2$ case; small $r$ (left), large $r$ (right). The external phases, which occur in the complementary components of the amoeba,
correspond to ``frozen'' or ``semi-frozen" configurations.}
\end{figure}

In the exterior of the amoeba we claim that the corresponding Gibbs measure concentrates 
on a single periodic configuration, depending on the exterior component as shown in the figure. This is clear for the large external regions,
corresponding to slopes $(0,0),(1,0),$ and $(0,1)$, where there is a unique configuration with that slope. The remaining external regions, which are bounded between two parallel ``tentacles'' of the amoeba, behave like the semi-frozen states of \cite{KOS}. For the region corresponding to slope $(0,\frac12)$ on the left panel, for example, the configuration must consist of only horizontal edges, of total density $1/2$; however the horizontal paths at even $y$-value and at odd $y$-value have different weights. The higher-weight horizontal path has exponentially larger weight than the other (as a function of system of size) and so in the limit we only see
the higher-weight horizontal paths. Likewise for the other semi-frozen regions, including the zig-zag semi-frozen state of slope $(\frac12,\frac12)$ 
on the right panel.

A plot of the free energy is shown in Figure \ref{F2X2smallr} for both small $r$ and for the large $r$ case.

\begin{figure}[htbp]
\begin{center}\includegraphics[width=3in]{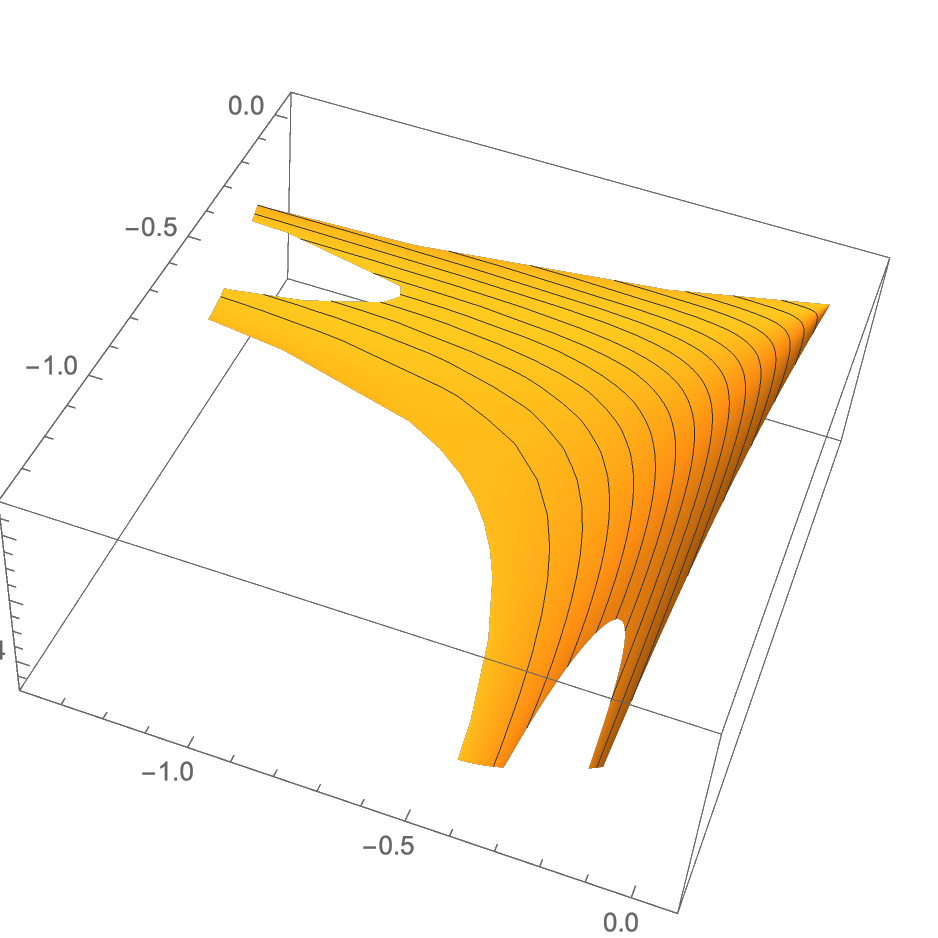}
\includegraphics[width=3.4in]{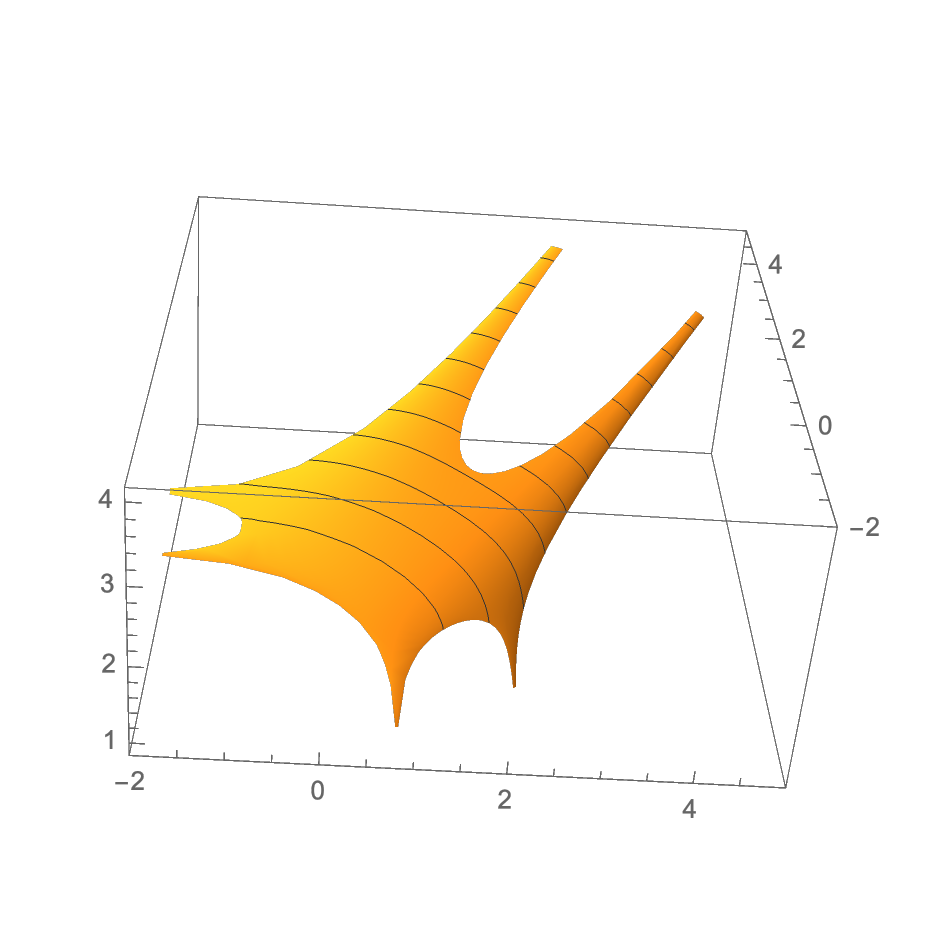}
\end{center}
\caption{\label{F2X2smallr}Free energy examples for $m_1=m_2=2$; small $r$ (left), large $r$ (right).}
\end{figure}

\subsection{Coexistence phase boundary}

In the small $r$ case, the surface tension is not strictly convex everywhere in $\N$. There is a 
``coexistence'' phase, $\N\setminus\bar\N'$, which is the complement of the closure
of the region on which $\sigma$ is strictly convex. The coexistence phase is bounded by the line $s+t=1$ and a convex curve, which we find here. 
The boundary of $\N'$ along this curve occurs when $\theta\to0$. 
In this case, we have (using (\ref{eq:slopesfromu}) with $u=Re^{i\theta}$, and $R\in(0,\infty)$ fixed) 
$$\lim_{\theta\to0}s=
\frac1{m_1}\sum_i \frac{\alpha_i^2}{R+\alpha_i^2}$$
and
$$\lim_{\theta\to0}t =\frac1{m_2}\sum_j \frac{\beta_j^2R}{1+\beta_j^2R}.$$

As $R\in(0,\infty)$ the curve $(s,t)$ describes a convex curve in $\N$ connecting $(1,0)$ to $(0,1)$,
see Figure \ref{neutral2X2} for an example. Thus $\N'$ is given explicitly as the region in the quadrant $\{ (s,t) \colon s\geqslant 0, t \geqslant 0\}$ below this curve.

\begin{figure}[htbp]
\centerline{\includegraphics[width=2.in]{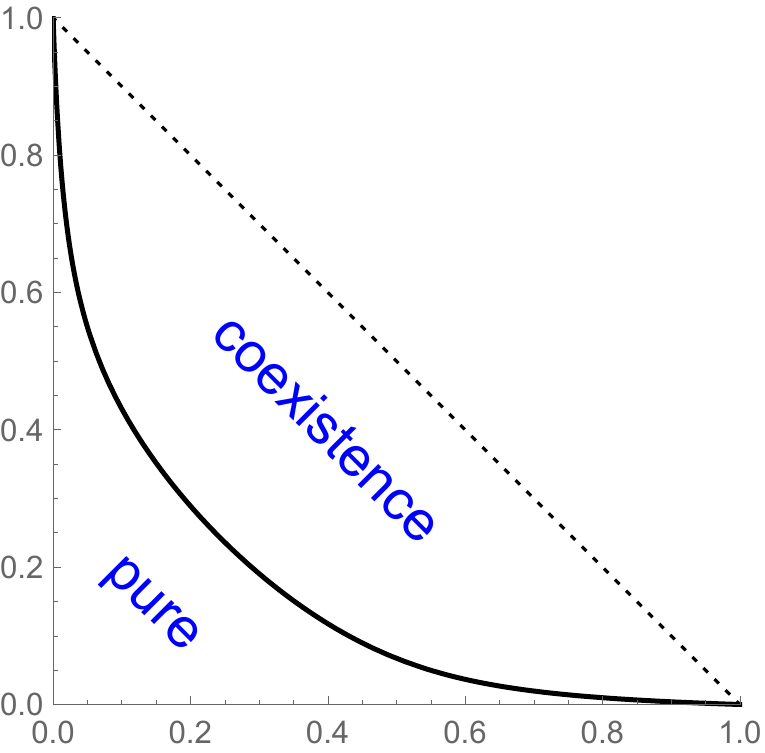}}
\caption{\label{neutral2X2}Coexistence phase boundary for $m_1=m_2=2$ and $\{\alpha_1,\alpha_2\}=\{\beta_1,\beta_2\}= \{.2,.9\}$.}
\end{figure}

\subsection{Surface tension}

We collect our main results about the structure of the surface tension in the following theorem.

\begin{theorem}
\label{thm:surfacetension}
a) [Small $r$ case] Consider a genus-zero five vertex model with an $m_1 \times m_2$ fundamental domain in the small $r$ case. The surface tension $\sigma$ is a convex function in $\N$; it is piecewise linear on $\D\N$ with slope discontinuities located at a subset of the points $(\frac{i}{m_1},0)$, $i=1,\ldots m_1-1$ and $(0,\frac{j}{m_2})$, $j=1,\ldots m_2-1$. It is strictly convex in $\mathring{\N'}$ where it has trivial potential,
\[ \sqrt{\det H_\sigma}=\frac{1}{\pi} (\arg u)^2,
\]
with a conformal parameter $u \in \H$ and $\sigma_{\N \setminus \bar \N'} \equiv 0$,
where $\N'$ is defined in the previous section.
The explicit expression for $\sigma$ is given by \eqref{doublesumsig}.

b) [Large $r$ case] Consider a genus-zero five vertex model with an $m_1 \times m_2$ fundamental domain in the large $r$ case. The surface tension $\sigma$ is a convex function in $\N$; it is piecewise linear on $\D \N$ with slope discontinuities located at a subset of the points $(\frac{i}{m_1},0)$, $i=1,\ldots m_1-1$ and $(0,\frac{j}{m_2})$, $j=1,\ldots m_2-1$, as well as at $(\frac12,\frac12)$. It is strictly convex in $\mathring{\N}$ and has trivial potential,
\[ \sqrt{\det H_\sigma}=\frac{1}{\pi} (2\pi -\arg u)^2,
\]
with a conformal parameter $u \in \H$.
The explicit expression for $\sigma$ is given by \eqref{sigmabigrgeneral}.

\end{theorem}

\section{Limit shapes}
\label{limitshapesection}

\subsection{Variational principle}
The limit shape problem is defined as follows.
Let $U\subset\R^2$ a bounded domain with piecewise smooth boundary and let $h_0:\partial U\to\R$ be a continuous function.
For $\eps>0$ small let $U_\eps = \eps\Z^2\cap U$. Let $\Omega(U_\eps)$ be the space of MNLP configurations on $U_\eps$ whose rescaled boundary height function $\eps H|_{\D U_{\eps}}$ approximates $h_0$ uniformly in the limit $\eps\to0$. 
We suppose $\Omega(U_\eps)$ is nonempty. 
On $U_\eps$ put $(m_1,m_2)$-periodic vertex weights $r_{ij}=\alpha_i\beta_j$,
and let $\mu_\eps$ be the associated probability measure on $\Omega(U_\eps)$. 
The limit shape problem asks about the typical height of a $\mu_\eps$-random configuration. Under appropriate hypotheses, the rescaled height functions concentrate onto a deterministic surface, called the limit shape.
Such a limit shape is determined by the surface tension $\sigma$ associated to this $(m_1,m_2)$-periodic model, where $\sigma$ is given by Theorem \ref{thm:surfacetension}. 

In this context the variational principle of \cite{CKP} takes the following form (see \cite[Corollary 4.11]{LT} and also \cite{dGKW}).
\begin{theorem}[\cite{CKP,LT}]\label{varthm}
Consider the variational problem of minimizing the surface tension integral
\begin{equation}
\label{eq:variationalproblem}
\min_h \iint_{U} \sigma(\nabla h)\,dx\,dy ~~~~~h|_{\partial U} = h_0,
\end{equation}
among all Lipschitz competitors $h$ with $\nabla h \in \N$ a.e.
For any $\delta>0$, with probability tending to $1$ as $\eps\to0$ a $\mu_\eps$-random configuration will have rescaled height function $\eps H$ lying uniformly within $\delta$ of 
one of the minimizers of \eqref{eq:variationalproblem}.
\end{theorem}

\begin{rmk}
Usually variational principles as above require strict convexity of the surface tension to ensure the required concentration estimates. Corollary 4.11 of \cite{LT} that we use here works without the assumption of strict convexity -- but in this case one can only deduce a weaker form of concentration: the random surface concentrates on the \emph{set} of minimizers. (Another common assumption, which is \emph{stochastic monotonicity}, implies strict convexity of the surface tension \cite{LT}. Therefore in the small $r$ case the model is not stochastically monotone. In the large $r$ case one can show that the model is monotone but we don't use this fact since we establish strict convexity of surface tension by other means.)

In the large $r$ case, $\sigma$ is strictly convex in $\mathring{\N}$ and there is a unique minimizer in \eqref{eq:variationalproblem}, see Proposition 4.5 in \cite{DSS}; a limit shape forms in the entire region $U$. 

In the small $r$ case, however, $\sigma$ fails to be strictly convex and there need not be a unique minimizer.
However, since the surface tension is convex the set of minimizers forms a convex set and, moreover all minimizers
coincide on the \emph{repulsive region}  $U' \subset U$ where (for any minimizer) $\nabla h \in \mathcal{N'}$ -- the pure phase of the surface tension.
The complement of the repulsive region is a region of \emph{non-uniqueness}: the minimizer has non-unique extensions into this region.
The variational principle ensures the formation of a (deterministic) limit shape only in the subregion $U'$. A priori, it could be that by resampling the random surface it will concentrate around \emph{another} minimizer and thus causing macroscopic fluctuations on the non-uniqueness region.
Indeed, we conjecture that this is the case: there is no limit shape on the non-uniqueness region, the surface remaining random in the limit of small lattice spacing.
\end{rmk}

\begin{rmk}
While we don't consider the $r_v\equiv1$ case explicitly, this can be recovered as a limiting case. The model in this case, however, is determinantal and equivalent to a dimer model: a staggered lozenge tiling model. The surface tension has trivial potential, with $\sqrt{\det H_\sigma} \equiv \pi$ in $\mathring{\N}$ by \cite{KOS}. The results of \cite{KO1,ADPZ} apply for the limit shapes which arise.
\end{rmk}

\subsection{Darboux integrability}

A PDE is \emph{integrable in the sense of Darboux} if there is an
explicit parameterization of solutions in terms of complex analytic functions. 
Concretely, this means given a PDE $F(u,u_{x_i},u_{x_ix_j},\dots)=0$ for a function $u=u(x_1,\dots,x_k),$ 
where $F$ depends on finitely
many partial derivatives of $u$, the PDE has all solutions of the form 
$u=G(f,f',\int\!f,\dots)$ for a fixed function $G$ depending on a finite number of derivatives and antiderivatives of
an \emph{arbitrary} complex-analytic function $f$  (or possibly several such functions). 
A classical example is the minimal surface equation and its Weierstrass-Enneper parameterization of solutions
\cite{Spivak4}.

Let $h$ be one of the minimizers of the variational problem \eqref{eq:variationalproblem}. We say that $(x_0,y_0) \in U$ is in the {\em liquid region} $\mathcal{L}$ if $h$ is  $C^1$ in a neighbourhood of $(x_0,y_0)$ and $\nabla h(x_0,y_0)$ is in $\mathring{\N'}$. (Recall that $\N'=\N$ in the large $r$ case.)
Since $\sigma$ is smooth and strictly convex on $\N'$, by \cite{Morrey}, $h$  is smooth and satisfies the Euler-Lagrange equation in the liquid region
\begin{equation}
\label{eq:EL} 
\mbox{div}(\nabla \sigma \circ \nabla h)=0 \quad \text{in $\mathcal{L}$}.
\end{equation}
The liquid part of the limit shape is the surface $\{ (x,y,h(x,y)) \colon (x,y) \in \mathcal{L} \} \subset \R^3$. 

Since by Theorem \ref{thm:surfacetension} $\sigma$ has trivial potential, we can use Theorem 4.1 of \cite{KP1} to give parametrizations of limit shapes in terms of harmonic functions. 
The tangent plane to the limit shape at $(x_0,y_0,h(x_0,y_0))$ for a point $(x_0,y_0) \in \mathcal{L}$ is the plane 
\begin{equation}
\label{eq:tangentplane}
 P_{(x_0,y_0)}=\{ x_3= \nabla h(x_0,y_0) \cdot (x,y) + h(x_0,y_0) - \nabla h(x_0,y_0) \cdot (x_0,y_0) \},
\end{equation}
where we use coordinates $(x,y,x_3)$ in $\R^3$.

\begin{theorem}\label{tangentthm}
As $(x_0,y_0) \in \mathcal{L}$ varies over a component of the liquid region the tangent planes in \eqref{eq:tangentplane} can be parametrized by a complex coordinate
$\zeta$:
\[ \nabla h(x_0,y_0)=(s(u),t(u)), \quad\quad\quad  h(x_0,y_0) - \nabla h(x_0,y_0) \cdot (x_0,y_0)=G(\zeta)/\theta,
\]
where $u=u(\zeta)$ is complex-analytic, $s(u)$ and $t(u)$ are given by \eqref{eq:slopesfromu},
$G(\zeta)$ is harmonic and 
$\theta(u)=\begin{cases}\arg(u)&\text{small $r$ case}\\
2\pi-\arg(u)&\text{large $r$ case}\end{cases}$.
\end{theorem}

\begin{proof}
Let $\mathcal{L}_0$ be a component of the liquid region. For $(x,y) \in \mathcal{L}_0$, $\nabla h(x,y)\in \mathring{\N'}$.  Since $u\mapsto(s(u),t(u))$ is a bijection $\H\to \mathring{\N'}$ (see paragraph following (\ref{eq:slopesfromu})), we can define the map $u \colon \mathcal{L}_0 \to \H$ by 
\[ \nabla h(x,y)=(s(u),t(u)).\]
The component $\mathcal{L}_0$ has intrinsic complex structure $\zeta$ (see equation (10) of \cite{KP1}), and 
the map $u=u(\zeta)$ is holomorphic. 

According to Theorem \ref{thm:surfacetension}, $\sigma$ has trivial potential with $(\det H_\sigma)^{1/4}=\pi^{-1/2} \theta$. This implies, by Theorem 4.1 of \cite{KP1}, that
\[ G(\zeta):=\theta(h(x,y) - \nabla h(x,y) \cdot (x,y)), \quad (x,y) \in \mathcal{L}_0
\]
is a harmonic function.
\end{proof}

Theorem 4.1 of \cite{KP1} also implies that $s \theta, t \theta$ and $\theta$ are harmonic functions of $\zeta$. 
As a consequence we arrive at the following equation.
\begin{corollary} \label{cor:darboux}
In any component $\mathcal{L}_0$ of $\mathcal{L}$, we have the equation
\begin{equation}
\label{eq:darboux}
(s \theta)_\zeta x + (t \theta)_\zeta y+G_\zeta-\theta_\zeta h(x,y)=0.
\end{equation}
Since $s \theta$, $t \theta$, $G$ and $\theta$ are all harmonic functions of $\zeta$, thus the complex $\zeta$-derivatives appearing in \eqref{eq:darboux} are all holomorphic functions of $\zeta$.
\end{corollary}

\begin{proof}
Since $\nabla h=(s,t)$, we have
\[ h_\zeta=s x_\zeta+t y_\zeta=(sx+ty)_\zeta -s_\zeta x - t_\zeta y,
\]
or
\[ s_\zeta x+ t_\zeta y +(h-(sx+ty))_\zeta=0.
\]
Rewriting in terms of ratios we get
\be\label{Gderiv} ( s\theta/\theta)_\zeta x + (t\theta/\theta)_\zeta y+ (G/\theta)_\zeta=0.
\ee
When expanded this is
\[ \frac{1}{\theta} \left( (s \theta)_\zeta x + (t \theta)_\zeta y +G_\zeta - \theta_\zeta (sx+ty+G/\theta) \right)=0.
\] 
Finally, notice that 
\begin{equation}
\label{eq:tangent}
sx+ty+G/\theta=h.
\end{equation}
\end{proof}

This corollary gives us the following form of Darboux integrability.
Near a noncritical point of $u(\zeta)$ we can choose $u=\zeta$ as a local coordinate.
The real and imaginary part of (\ref{Gderiv}) give two simultaneous 
linear equations for $x,y$ as functions of the (known) functions $s(u),t(u),\theta(u)$ and the
(arbitrary) harmonic function $G(u)$.
Solving gives $(x,y)=(x(u),y(u))$ as functions of $u$ and then from \eqref{eq:tangent} (or (\ref{eq:darboux})) 
we have
our surface $(x,y,h(x,y))$ parameterized as a function of $u$.

Locally there is no restriction on $G$; however to find the limit shape for a Dirichlet problem of the type in Theorem \ref{varthm} (that is, when we want $h$ to be single-valued function of $(x,y)$), we need an extra condition on $G$, that $G_\zeta$ is zero at critical points of $u=u(\zeta)$; see \cite{KP3}.

\subsection{Large $r$ example}

In the examples below, we consider some unbounded domains $U$, whose limit shapes
are somewhat simpler than for the simplest bounded domains, because we can take $u(\zeta)=\zeta$. These are to be understood as limiting cases of exhaustion by bounded domains. We sketch the derivation of the solution (which is based on \cite{KP3}) and give the final formula for the associated function $G(u)$ and an associated figure. We verify that $G(u)$ defines the limit shapes for the problems under consideration.

Figure \ref{2by2largerthumb} shows an example limit shape of a \emph{semi-boxed plane partition},
for a $2\times 2$ fundamental domain with $(\alpha_1,\alpha_2) = (\beta_1,\beta_2)= (2,5/4)$.
The ``semi-boxed plane partition" is an MNLP configuration in an infinite region $\{x\ge0,y\ge0,|x-y|<n\}$, as in Figure \ref{2by2largerthumb}. It can be considered a limit of boxed plane partitions in an $n\times n\times m$ box when sidelength $m$ tends to infinity. 
The rescaled region is $$U=\{x>0, y>0, |x-y|<1\},$$ with 
boundary height function $h_0(x,y)=0$ along the $x$-axis and $y$-axis, $h_0(x,y)=y$ along the line $x-y=1$ and $h_0(x,y)=x$ along the line $x-y=-1$. The liquid region is bounded in this case, even though $U$ is unbounded. Most of the limit shape 
consists of an unbounded facet of slope $(1/2,1/2)$. The limit shape is found using Theorem \ref{tangentthm}. 
It is the envelope of the planes $P_u = \{x_3=sx+ty+c\}$,
where $s=s(u),t=t(u)$ are functions of the parameter $u\in\H$ given by (\ref{zwfromu}), and 
$c=G(u)/\theta$ where the harmonic function $G(u)$ is determined by its boundary values. We claim that for this boundary 
we have
$$G(u) = -\pi-\arg(u),$$
that is, $G(u)=-\pi$ for positive $u$ and $G(u)=0$ for negative $u$.
To see this, note that we expect six complementary regions to the limit shape, and the tangent planes in these regions
are respectively, starting clockwise from the region containing the origin, $x_3=0, x_3=\frac12x, x_3=x,x_3=\frac12x+\frac12y-\frac12,x_3=y,x_3=\frac12y$. As $u$ runs along $\R$, $s(u),t(u)$ and $c(u)$ are step functions which take values
$$(s,t,c)=\begin{cases}

(0,1,0)&u<-4\\
(0,1/2,0)&-4<u<-\frac{25}{16}\\
(0,0,0)&-\frac{25}{16}<u<\frac{16}{25}\\
(\frac12,0,0)&\frac{16}{25}<u<-\frac{1}{4}\\
(1,0,0)&-\frac{1}{4}<u<0\\
(\frac12,\frac12,-\frac12)& 0<u
\end{cases}.$$
Using (\ref{zwfromu}) and (\ref{XYstformulaslarge}) these correspond exactly to the desired boundary values of $(s,t,c)$ for the limit shape.
Therefore $G(u)$ defines the correct tangent planes on the limit shape along the boundary.
The fact that $G(u)$ is harmonic implies that the envelope of these planes for $u\in \H$ solves the Euler-Lagrange equation in the non-faceted region $\mathcal{L}$ specified below.
This construction thus defines the correct limit shape in the sense that the defined height function $h$ solves the Euler-Lagrange equation in $\mathcal{L}$, and extended linearly on the six complimentary facets matches up with the required boundary values $h_0$. In our example one can verify that this is indeed the unique minimizer with a similar argument to \cite[Section 8]{ADPZ} as follows.

The boundary of $\mathcal{L}$ is an envelope of lines parametrized by $u \in \mathbb{R} \cup \{\infty\}$.
\begin{equation}
\label{eq:tangentlines}
  s_u x +t_u y+ (G/\theta)_u=0,
\end{equation}
which defines $\partial \mathcal{L}$ as locally convex except at two cusp points and tangent to the four boundary lines of $U$. Each of the six components of $U \setminus \mathcal{L}$, corresponding to the different facets of the limit shape, can be covered by a foliation of tangent lines of \eqref{eq:tangentlines}. By construction
\[ \nabla h (x,y) = (s(u),t(u)) \quad \text{and } \quad \Phi(x,y):= \nabla \sigma ( \nabla h(x,y)) = (X(u),Y(u))
\]
for $(x,y) \in \mathcal{L}$. 
We now extend the vector field $\Phi$ continuously to the facets as constant following the foliations by tangent lines.
In terms of the $u$ variable, this extension defines a map $u \colon U \to \overline {\mathbb{H}}$ such that
\[ \frac{u_x}{u_y}=\frac{s_u}{t_u}=-\frac{Y_u}{X_u}.
\]
This means that $\mathrm{div}\, \Phi(x,y)=2 \Re (X_u u_x+ Y_u u_y)=0$ everywhere in $U$. Observe that by construction $\Phi(x,y) \in \partial \sigma (\nabla h(x,y))$ for every $(x,y) \in U$. Here the notation $\partial \sigma$ denotes the subgradient of the convex function $\sigma$.
The existence of a vector field $\Phi$ with these properties ensures that $h$ is the minimizer, see \cite[Proposition 8.1]{ADPZ}. We note that Proposition 8.1 is stated for dimer-specific surface tensions $\sigma$ in \cite{ADPZ} but it applies equally well in our setting as only the convexity of $\sigma$ is essential in the proposition.

\begin{figure}[htbp] \label{fig:semiboxed}
\centerline{\includegraphics[width=2.5in]{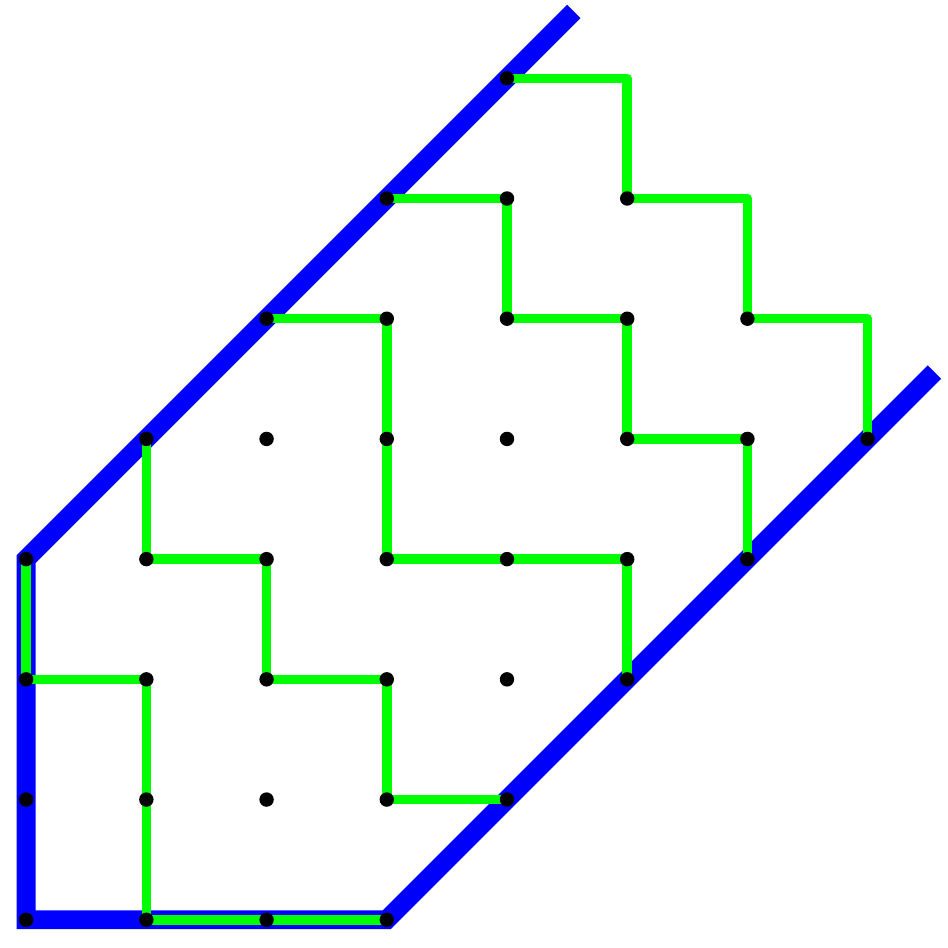}\hskip1in\includegraphics[width=2.5in]{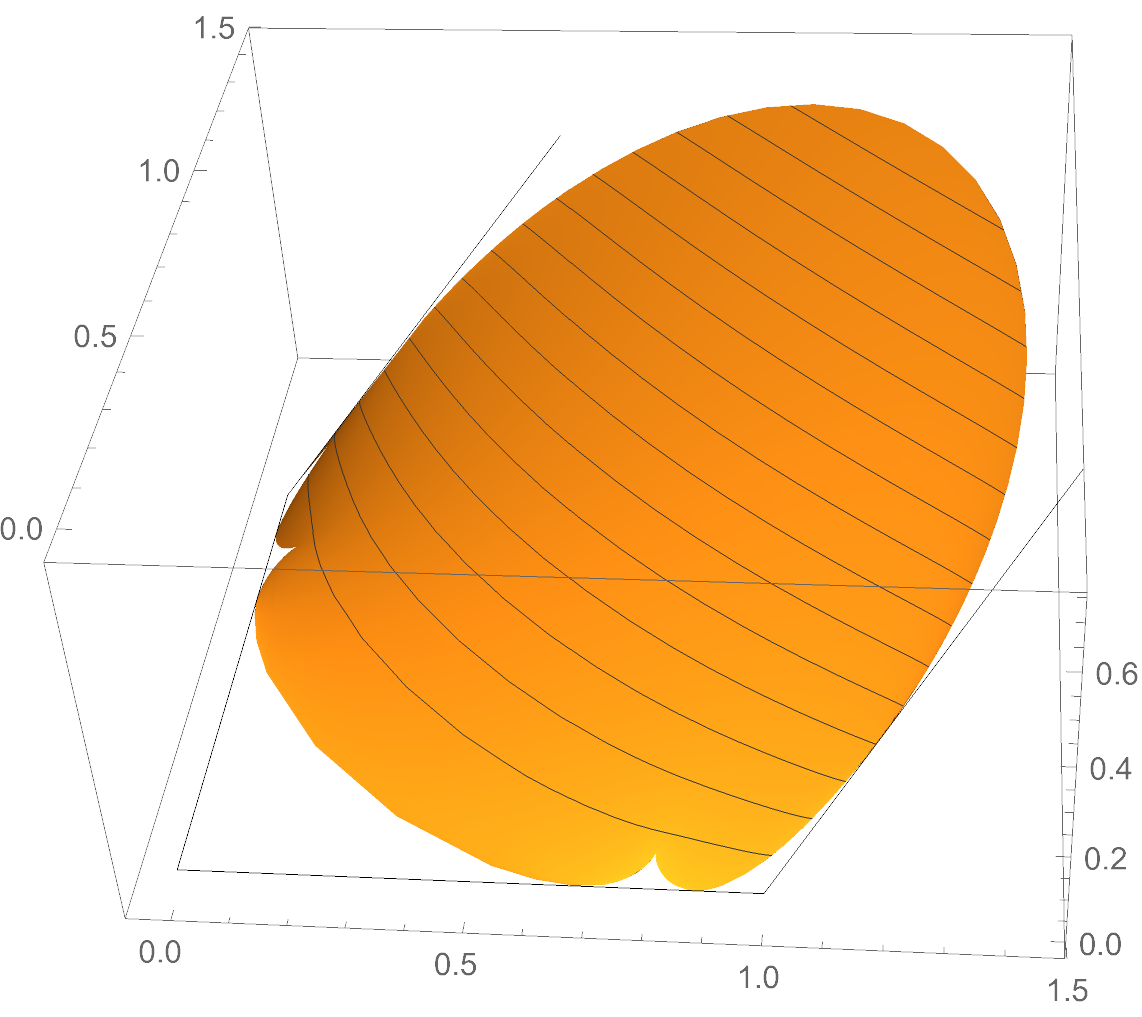}}
\caption{\label{2by2largerthumb} Left panel: the ``semi-boxed plane partition". For large $r$ our results show that a $5$-vertex configuration is eventually (as $x+y\to\infty$) a completely packed set of zig-zag paths. Right panel: Limit shape for the semi-boxed plane partition,
for a $2\times2$ fundamental domain, large $r$ case. The liquid part of the limit shape is plotted. The limit shape is linear on each component of the complement of this shape.
} 
\end{figure}

\begin{figure}[htbp]
\centerline{\includegraphics[width=3.in]{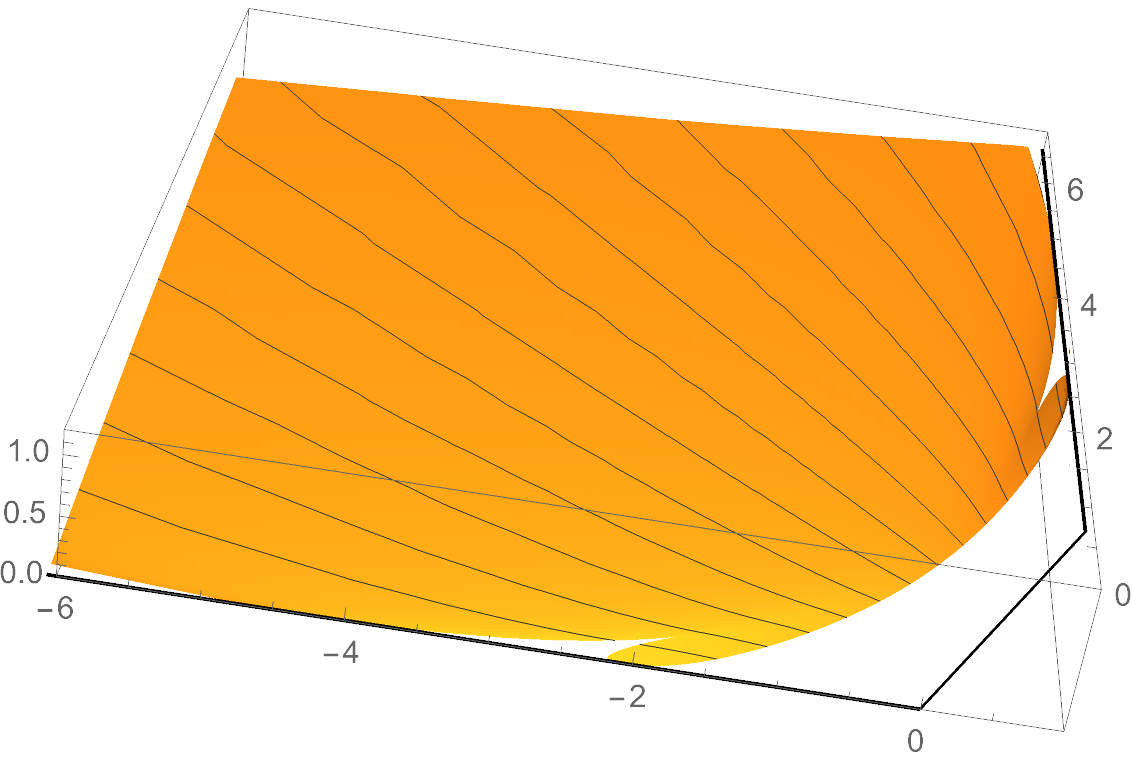}}
\caption{\label{semibpp2by2}One of the limit shapes for another type of semi-boxed plane partition boundary conditions,
for the five-vertex model in a $2\times2$ fundamental domain, small $r$ case. The liquid part of the limit shape is plotted.} 
\end{figure}

\subsection{Small $r$ example}

Figure \ref{semibpp2by2} shows an example limit shape of a different semi-boxed plane partition,
which is the limit of the MNLP configuration in an $n_1\times n_2\times n_3$ box when sidelengths $n_2,n_3$ tend to infinity. 
The rescaled region is the region $$U=\{x<1, y>0, x-y<1\},$$ 
with boundary height function $0$ along the $x$-axis, $1$ along the $y$ axis and $y$ along the diagonal edge. 
The limit shape shown is for a $2\times 2$ fundamental domain with $(\alpha_1,\alpha_2) = (\beta_1,\beta_2)= (4/5,1/4)$.
Interestingly, there is a one-parameter family of limit shapes in this region,
each of which has liquid region of roughly parabolic shape for large $x,y$. The direction of the
axis of this parabola is a free parameter.  
The relative rates at which $n_2,n_3$ go to infinity determine this parameter.
(The same phenomenon already appears for the lozenge model in which case the frozen boundary is an exact parabola.)

We can describe the limit shape as the envelope of the planes $P_u = \{x_3=sx+ty+c\}$,
where $s=s(u),t=t(u)$ are functions of the parameter $u$ given by (\ref{zwfromu}), and 
$c=G(u)/\theta$ where $G(u)$ is determined, up to a single parameter $a$, by the boundary values.
Here $u\in(0,\infty)$ on the coexistence boundary, and $u=-\alpha_1^2,-\alpha_2^2$ at the tangency
points with the line $x=1$, and $u=-\beta_1^{-2},-\beta_2^{-2}$ at the tangency points with the line $y=0$.
The value of $u$ at the point at infinity in the limit shape is defined to be $u=a$, where $a$ can range over
the interval $(-\alpha_1^2,-\beta_2^{-2})$. 
In this case for $u\in \R$ we have 
$$
G(u) = \begin{cases}\frac{\pi}2&u\in[-\alpha_2^2,-\alpha_1^2]\\
\pi&u\in[-\alpha_1^2,a]\\
0&\text{else.}\\
\end{cases}
$$
The harmonic extension to $u\in\H$ is then
$$G(u) = \arg(\frac{u-a}{\sqrt{(u+\alpha_1^2)(u+\alpha_2^2)}}).$$

Note that the limit shape shown does not touch the diagonal boundary edge. At the corners of the boundary there are two thin facets along the $x$- and $y$-axes with slopes $(1,0)$ and $(0,1)$ respectively (not depicted in the picture). The remaining part near the diagonal boundary edge is a region of non-uniqueness in terms of the variational problem:
see the discussion following Theorem \ref{varthm}.

\section{Open questions}\label{openq}

\begin{question}
In the coexistence region in the boxed-plane partition limit (for example) is the configuration random in the limit?
\end{question}

\begin{question}
Is there a natural ``higher genus'' model? That is, are there choices of staggered weights for the five-vertex model which produce ``gas'' phases (with multiply-connected amoebae) as in \cite{KOS} and still have trivial potential?
\end{question}

\begin{question}
Is there a duality between the small $r$ and large $r$ cases? In Theorem \ref{thm:surfacetension} the angles $\theta=\pi^{1/2} (\det H_\sigma)^{1/4}$ for large $r$ and small $r$ cases sum to $2\pi$. 
If we think of the model in terms of interacting lozenges, the small $r$ case corresponds to repulsive, while the large $r$ case corresponds to attractive interaction (and angle $\pi$ corresponds to no interaction). \end{question}

\bibliographystyle{hplain}
\bibliography{5vtx8}
\end{document}